\documentclass[a4paper,11pt,oneside]{amsart}
\pdfoutput=1

\usepackage[utf8]{inputenc}
\usepackage{bm}
\usepackage{mathtools,amssymb}
\usepackage{esint}
\usepackage{hyperref}
\hypersetup{
    colorlinks=true,
    linkcolor=black,
    filecolor=black,      
    urlcolor=cyan,
    citecolor=black
    }
\usepackage{tikz}
\usepackage{dsfont}
\usepackage{relsize}
\usepackage{url}
\usepackage{xcolor}
\usepackage{graphicx}
\usepackage{mathrsfs}
\usepackage[shortlabels]{enumitem}
\usepackage{lineno}
\usepackage{amsmath}
\usepackage{enumitem}
\usepackage{amsthm} 
\usepackage{fullpage} 
\usepackage{verbatim}
\usepackage{dsfont}
\usepackage{nicefrac}

\allowdisplaybreaks

\mathtoolsset{showonlyrefs}

\graphicspath{{images/}}

\newtheorem{theorem}{Theorem}[section]
\newtheorem{lemma}[theorem]{Lemma}
\newtheorem{definition}[theorem]{Definition}
\newtheorem{corollary}[theorem]{Corollary}

\newtheorem{proposition}[theorem]{Proposition}
\newtheorem{question}[theorem]{Question}
\newtheorem{remark}[theorem]{Remark}

\title[The fractional $p\,$-biharmonic systems]{The fractional $p\,$-biharmonic systems: optimal Poincaré constants, unique continuation and inverse problems}

\author{Manas Kar}
\thanks{Indian Institute of Science Education and Research (IISER) Bhopal (\href{mailto:manas@iiserb.ac.in}{manas@iiserb.ac.in})}
\address{Indian Institute of Science Education and Research (IISER) Bhopal, India}
\email{manas@iiserb.ac.in}
\author{Jesse Railo}
\thanks{Department of Pure Mathematics and Mathematical Statistics, University of
Cambridge (\href{mailto:jr891@cam.ac.uk}{jr891@cam.ac.uk})}
\address{Department of Pure Mathematics and Mathematical Statistics, University of
Cambridge, Cambridge CB3 0WB, UK}
\email{jr891@cam.ac.uk}
\author{Philipp Zimmermann}
\thanks{Department of Mathematics, ETH Zurich (\href{mailto:philipp.zimmermann@math.ethz.ch}{philipp.zimmermann@math.ethz.ch})}
\address{Department of Mathematics, ETH Zurich, Z\"urich, Switzerland}
\email{philipp.zimmermann@math.ethz.ch}
\date{\today}

\newcommand{\C}{{\mathbb C}}
\newcommand{\R}{{\mathbb R}}
\newcommand{\K}{{\mathbb K}}
\newcommand{\Z}{{\mathbb Z}}
\newcommand{\N}{{\mathbb N}}

\newcommand{\schwartz}{\mathscr{S}}

\newcommand{\tempered}{\mathscr{S}^{\prime}}

\newcommand{\fraclaplace}{(-\Delta)^s}
\newcommand{\fourier}{\mathcal{F}}
\newcommand{\ifourier}{\mathcal{F}^{-1}}
\newcommand{\vev}[1]{\left\langle#1\right\rangle}

\newcommand{\norm}[1]{\lVert #1 \rVert}
\newcommand{\abs}[1]{\left\lvert #1 \right\rvert}
\DeclareMathOperator{\Div}{div} 
\DeclareMathOperator{\Nabla}{\overline{\nabla}} 
\DeclareMathOperator{\DivH}{\overline{div}} 
\DeclareMathOperator{\Lap}{\overline{\Delta}} 
\DeclareMathOperator{\supp}{supp} 

\newcommand\scalemath[2]{\scalebox{#1}{\mbox{\ensuremath{\displaystyle #2}}}} 

\begin{document}

\maketitle
\begin{abstract} This article investigates nonlocal, fully nonlinear generalizations of the classical biharmonic operator $(-\Delta)^2$. These fractional $p$\,-biharmonic operators appear naturally in the variational characterization of the optimal fractional Poincaré constants in Bessel potential spaces. We study the following basic questions for anisotropic fractional $p$\,-biharmonic systems: existence and uniqueness of weak solutions to the associated interior source and exterior value problems, unique continuation properties (UCP), monotonicity relations, and inverse problems for the exterior Dirichlet-to-Neumann maps. Furthermore, we show the UCP for the fractional Laplacian in \emph{all} Bessel potential spaces $H^{t,p}$ for any $t\in \R$, $1 \leq p < \infty$ and $s \in \R_+ \setminus \N$: If $u\in H^{t,p}(\R^n)$ satisfies $(-\Delta)^su=u=0$ in a nonempty open set $V$, then $u\equiv 0$ in $\R^n$. This property of the fractional Laplacian is then used to obtain a UCP for the fractional $p$\,-biharmonic systems and plays a central role in the analysis of the associated inverse problems. Our proofs use variational methods and the Caffarelli--Silvestre extension.
\end{abstract}

\setcounter{tocdepth}{1}
\tableofcontents

\section{Introduction}
\label{intro}

The classical $p$\,-biharmonic operator is given by
\begin{equation}
\label{eq: p biharmonic operator}
    (-\Delta)_p^2u = \Delta(|\Delta u|^{p-2}\Delta u)
\end{equation}
 which is a nonlinear, elliptic, fourth order generalization of the well-known biharmonic operator $(-\Delta)^2$. A general boundary value problem for the $p$\,-biharmonic operator could then be formulated as follows: Find a function $u\colon \Omega\to \R$ solving
 \begin{equation}
 \label{eq: boundary value problem p biharm eq}
     \begin{split}
    (-\Delta)_p^2u &= f(x,u)\quad \text{in}\quad \Omega, \\
    B_ju&=g_j\quad\quad \quad \,\text{on}\quad \partial\Omega
\end{split}
 \end{equation} 
 for $j=1,2$, where $\Omega\subset\R^n$ is some domain with sufficiently smooth boundary, $f\colon \Omega\times \R\to\R$ a possibly nonlinear function, $B_j$, $j=1,2$, are some boundary operators and $g_1,g_2\colon\partial\Omega\to \R$ are given boundary data. Typical boundary conditions which have been studied in the existing literature are the Navier boundary conditions $u=g_1,\,\Delta u=g_2$ on $\partial\Omega$, the Dirichlet boundary conditions $u=g_1,\,\partial_{\nu}u=g_2$ on $\partial\Omega$, where $\nu$ denotes the unit outer normal to $\partial\Omega$, or combinations of them which are called mixed Dirichlet--Navier boundary conditions. The regularity properties of biharmonic functions, that is solutions to  \eqref{eq: boundary value problem p biharm eq} with $f\equiv 0$, spectral properties of biharmonic operators, variational formulations and unique continuation principles have been studied extensively for $p=2$ (see e.g.~the articles \cite{Gazzola} for well-posedness, regularity properties, \cite{SpectraBiharm, Spectrum-Biharm} for spectral properties and \cite{ Colombini:Grammatico:1999,colombini:Koch:2010, Lin:Sei:Wang:2011} for the strong unique continuation properties for the fourth order elliptic equation). Moreover, in the article \cite{CaffBiharmObstacle} Caffarelli and Friedman studied the obstacle problem for the biharmonic operator. In recent years, many of these results have been extended to the case $p\neq 2$. For example in the works \cite{Spectrum-p-Biharm, Talbi2007OnTS} the authors analyzed the spectrum of the $p$\,-biharmonic operator and showed that the eigenvalue problem associated to \eqref{eq: p biharmonic operator}, namely problem \eqref{eq: boundary value problem p biharm eq} with $f(x,u)=\lambda |u|^{p-2}u$ and homogeneous Navier boundary conditions, have a simple, isolated least positive eigenvalue $\lambda_+>0$.  

This work is devoted to the study of the \emph{anisotropic fractional $p$\,-biharmonic operators}
\begin{equation}\label{eq:p-biharmonic operator}
    (-\Delta)_{p,A}^s u \vcentcolon = (-\Delta)^{s/2}(|A^{1/2}(-\Delta)^{s/2} u|^{p-2}A(-\Delta)^{s/2}u)
\end{equation}
where $1 < p < \infty$, $s > 0$ and $A\in L^{\infty}(\R^n;\R^{m\times m})$ is a symmetric, positive definite, uniformly elliptic matrix-valued function. We will simply call the operator \eqref{eq:p-biharmonic operator} as the \emph{fractional $p$-biharmonic operator} when $A=\mathbf{1}_m$ and often restrict to the truly fractional cases $s \in \R_+ \setminus \Z$ or $s \in \R_+ \setminus 2\Z$. The associated exterior value problem takes the form
\begin{equation}
    \label{eq: exterior value problem}
        \begin{split}
            (-\Delta)^s_{p,A}u&=f(x,u)\quad \text{in}\quad \Omega,\\
            u&=g\quad\quad \quad\,\,\, \text{in}\quad \Omega_e
        \end{split}
    \end{equation}
where $\Omega_e\vcentcolon = \R^n\setminus\overline{\Omega}$ is the exerior of $\Omega$, $f\colon\R^n\times\R\to\R$ is a possibly nonlinear function or $f\in(\widetilde{H}^{s,p}(\Omega;\R^m))^*$ (see Section~\ref{sec: preliminaries}), which models an interior source, and $g\in H^{s,p}(\R^n;\R^m)$ is the prescribed exterior value of $u$. As we will see later, in the cases $f\equiv 0$ (pure exterior value problem) or $f\in (\widetilde{H}^{s,p}(\Omega;\R^m))^*$ and $g\equiv 0$ (pure interior source problem), the solutions $u$ can be obtained by minimizing a related energy functional (called $p$\,-energy). The considered energy functional is similar to the one considered in the work \cite{RegularityHarmonicMaps} but there $s$ is fixed to the critical value $s=n/p$, $A=\mathbf{1}_m$ and the authors considered functions $u$ taking values in a closed Riemannian manifold $N\subset\R^m$.

Next, we describe our main contributions (the detailed discussion is given in Section~\ref{sec: main results}) and the structure of this article. We introduce the basic notation and functional setting used in this work in Section~\ref{sec: preliminaries}. We start in Section~\ref{sec: varioational char} by showing that the fractional $p$\,-biharmonic operator \eqref{eq:p-biharmonic operator} studied in this work naturally appears in the variational characterization of the optimal fractional Poincaré constants in Bessel potential spaces. In fact, we prove that there is a function $u\in H^{s,p}(\R^n)$ whose $p$\,-energy coincides with $C^{-p}_*$, where $C_*$ is the optimal Poincar\'e constant, and it solves \eqref{eq: exterior value problem} with $f(x,u)=\lambda|u|^{p-2}u$ for some $\lambda>0$.
In Section~\ref{sec: existence theory}, we establish the existence and uniqueness of weak solutions to the (anisotropic) fractional $p$\,-biharmonic systems \eqref{eq: exterior value problem} in the two mentioned limiting cases of pure exterior values and interior sources. We define the exterior Dirichlet-to-Neumann (DN) maps related to the anisotropic fractional $p$\,-biharmonic operators in Section~\ref{subsec: trace space and DN maps}. We study unique continuation properties (UCP) of these nonlinear, nonlocal operators, in Section~\ref{sec: UCP}. Finally, in Section~\ref{sec: inverse problems main}, we establish monotonicity relations for the fractional $p$\,-biharmonic operators and uniqueness results for the related exterior data inverse problems in the presence of monotonicity assumptions for certain conformal coefficients of a priori known anisotropy.

\subsection*{Acknowledgements} The authors are grateful to Giovanni Covi, María Ángeles García-Ferrero and Angkana Rüland for helpful discussions on the UCP for fractional Laplacians and suggesting a proof of Theorem~\ref{UCP} for $0 < s < 1$, which we partially follow. The authors thank Shiqi Ma, Mikko Salo and Joaquim Serra for helpful discussion related to this article. M.K. was supported by MATRICS grant (MTR/2019/001349) of SERB. J.R. was supported by the Vilho, Yrjö and Kalle Väisälä Foundation of the Finnish Academy of Science and Letters.

\section{Main results of the article and comparison to the literature}
\label{sec: main results}

In this section, we state and discuss the main results obtained in this work. We also briefly compare our results to the existing literature on the way. We refer to the following books on the basics of the fractional Laplacian, fractional Sobolev spaces and their applications \cite{Singular-Integrals-Stein,Triebel-Theory-of-function-spaces}.

\subsection{On the optimal fractional Poincar\'e constants} For any bounded open set $\Omega \subset \R^n$, $1 < p <\infty$ and $s \geq 0$ there exists a constant $C(n,p,s,\Omega) > 0$ such that
\begin{equation}\label{eq:basicPoincareIntro}
    \norm{u}_{L^p(\Omega)} \leq C\norm{(-\Delta)^{s/2}u}_{L^p(\R^n)}
\end{equation}
for all $u \in C_c^\infty(\Omega)$ (see e.g. \cite[Lemma 3.3]{ARS21-OnFractVersionMurat} or \cite[Lemma 5.4]{RZ2022unboundedFracCald}). Later on we will refer to \eqref{eq:basicPoincareIntro} as the fractional Poincaré inequality. Note that we do not require any boundary regularity of the domain $\Omega$, which is similar as for example in the classical Sobolev embedding theorem for $W^{k,p}_0(\Omega)$-functions (see e.g.~\cite[Theorem~A.5]{Variational-Methods}) and hence for the Poincar\'e inequality in these spaces. We establish a variational characterization of the optimal fractional Poincaré constant in \eqref{eq:basicPoincareIntro} when $1 < p < \infty$ and $s > 0$. This characterization is directly related to the fractional $p$\,-biharmonic operator given in \eqref{eq:p-biharmonic operator} and in part motivates to investigate properties of the fractional $p$\,-biharmonic operators and their other relations with the fractional Laplacians. Further references on \eqref{eq:basicPoincareIntro} and the higher order fractional Laplacians can be found in \cite{RZ2022unboundedFracCald}.

One important application of the fractional Poincar\'e inequalities is to show well-posedness results for certain nonlocal partial differential equations (PDEs). More precisely, these inequalities allow to obtain coercivity estimates for the weak formulations of some nonlocal operators which together with the Lax--Milgram theorem prove existence of unique solutions (see e.g.~\cite{DirichletNonlocalOperators,RZ2022unboundedFracCald,RosOton16-NonlocEllipticSurvey}). Moreover, we point out that the constant in the stability estimates of the obtained unique solutions via the Lax--Milgram theorem depend linearly on the Poincar\'e constant which further motivates the study of the optimal fractional Poincar\'e constants. 

The standard examples of nonlocal PDEs are the uniformly elliptic integro-differential operators which have the form
\[
    Lu(x)\vcentcolon = p.v.\int_{\R^n}(u(x)-u(y))K(x-y)\,dy,
\]
where the kernel $K\colon\R^n\to \R$ satisfies
\[
    K(x)\geq 0,\quad K(-x)=K(x),\quad \frac{\lambda}{|x|^{n+2s}}\leq K(x)\leq\frac{\Lambda}{|x|^{n+2s}}
\]
for all $x\in\R^n$ with $0<\lambda<\Lambda$, $0<s<1$ and $p.v.$ stands for the Cauchy principal value. These operators naturally show up as the infinitesimal generators of stable L\'evy processes or more precisely the associated semigroups. A particular simple and well-behaved uniformly elliptic integro-differential operator is the fractional Laplacian $(-\Delta)^s$ for $0<s<1$, which corresponds to a stable, radially symmetric L\'evy process, and its higher order generalization to $s\in \R_+$ (see e.g.~\cite{DirichletNonlocalOperators,RosOton16-NonlocEllipticSurvey}). The cases when $p \neq 2$ appear naturally in the studies of nonlinear PDEs and the standard example is the fractional $p$\,-Laplacian which is usually defined as
\begin{equation}\label{eq: fracpLapDef}
    (-\Delta_p)^su(x)\vcentcolon =C_1p.v.\int_{\R^n}\frac{|u(x)-u(y)|^{p-2}(u(x)-u(y))}{|x-y|^{n+sp}}\,dy
\end{equation}
for $0<s<1$ and some normalizing constant $C_1>0$. The constant $C_1$ can be chosen in such a way that $(-\Delta_p)^su\to -\Delta_pu$ as $s\uparrow 1$ and $(-\Delta)^s_pu\to (-\Delta)^su$ as $p\downarrow 2$ for sufficiently smooth functions $u$, where $\Delta_p$ denotes the $p$\,-Laplacian defined by $\Delta_p u\vcentcolon = \text{div}(|\nabla u|^{p-2}\nabla u)$ (see e.g.~\cite{ThreeRep} 
and references therein). The fractional Poincar\'e inequality, the Sobolev embedding theorem and inequalities for closely related operators have been studied extensively in the literature (see e.g. \cite{PBL18-Opt-Prob-First-Eigenvalues-pFracLap,DirichletNonlocalOperators,HYZ12-FracGaglNirenHardy,LL14--Fractional-eigenvalues,MS21-Best-frac-p-Poincare-unboundedDoms}).

For any $1< p<\infty$, $s\geq 0$ and an open bounded set $\Omega\subset\R^n$, we define the set
\[
    \mathcal{M}_p=\{u\in\widetilde{H}^{s,p}(\Omega);\,\|u\|_{L^p(\R^n)}=1\}\subset \widetilde{H}^{s,p}(\Omega)
\] 
and the energy functional
\[
     \mathcal{E}_p\colon\mathcal{M}_p\to \R_+,\quad \mathcal{E}_p(u)=\int_{\R^n}|(-\Delta)^{s/2}u|^p\,dx.
\]
We have obtained the following result on the optimal fractional Poincar\'e constants:
\begin{theorem}
\label{Theorem: variational characterization of fractional poincare constant}
    Let $1< p<\infty$, $s> 0$ and $\Omega\subset\R^n$ be an open bounded set and denote by $C_*=C_*(n,p,s,\Omega)>0$ the optimal fractional Poincar\'e constant. Then the following statements hold:
    \begin{enumerate}[(i)]
        \item\label{item 1 var char} The constant $C_*$ satisfies
        \[
            C_*^{-p}=\inf_{v\in\mathcal{M}_p}\mathcal{E}_p(v).
        \]
        \item\label{item 2 var char} There exists a minimizer $u\in\mathcal{M}_p$ with $\lambda_{1,s,p}\vcentcolon =\mathcal{E}_p(u)>0$. 
        \item\label{item 3 var char} Any minimizer $u\in \mathcal{M}_p$ solves the following Euler--Lagrange equation
        \begin{equation}
        \label{eq: EL eq}
              \int_{\R^n}|(-\Delta)^{s/2}u|^{p-2}(-\Delta)^{s/2}u (-\Delta)^{s/2}v\,dx=\lambda_{1,s,p}\int_{\R^n}|u|^{p-2}uv\,dx
        \end{equation}
        for all $v\in \widetilde{H}^{s,p}(\Omega)$.
        \item\label{item 4 var char} If $0\neq v\in\widetilde{H}^{s,p}(\Omega)$, $\mu\in\C$ satisfy
        \begin{equation}
        \label{eq: identity 4}
             \int_{\R^n}|(-\Delta)^{s/2}v|^{p-2}(-\Delta)^{s/2}v (-\Delta)^{s/2}w\,dx=\mu\int_{\R^n}|v|^{p-2}vw\,dx
        \end{equation}
        for all $w\in C_c^{\infty}(\Omega)$, then $\mu \in \R$ and there holds $\mu\geq \lambda_{1,s,p}$.
    \end{enumerate}
\end{theorem}

Theorem \ref{Theorem: variational characterization of fractional poincare constant} is completely analogous to the well-known classical result, which connects the optimal Poincar\'e constant in $\|u\|_{L^p(\Omega)}\leq C\|\nabla u\|_{L^p(\Omega)}$ and the $p$\,-Laplace operator $\Delta_p$ (see e.g.~ \cite{MR1968344,Esposito:Nitsch:Trombetti:2013,EigenPLap2,EigenPLap}). We remark that the fractional $p$\,-Laplacian \eqref{eq: fracpLapDef} also shows up similarly in the variational characterization of the optimal Poincar\'e constants in Slobodeckij--Gagliardo spaces $W^{s,p}(\Omega)$ (see e.g. \cite[Section~3]{MR3264796} and \cite{MS21-Best-frac-p-Poincare-unboundedDoms}). In the setting of Slobodeckij--Gagliardo spaces, higher order eigenvalues of the fractional $p$\,-Laplacians and fractional capacities of sets are also studied recently (see e.g.~\cite{Brasco:Parini:2016,FracpLaplacian, DELPEZZO2020111479}). The authors are not aware of similar studies for the Bessel potential seminorms (i.e. $L^p$ norms of the fractional Laplacians) or for the fractional $p$\,-biharmonic operators. In part, these connections and analogies make it tempting to study the properties of the fractional $p$\,-biharmonic operators further.

\subsection{On the unique continuation properties of the fractional Laplacians and $p$\,-biharmonic systems} Unique continuation properties for (elliptic) operators have a long history \cite{A57, AKS62,C39, KT08} and dates back at least to Riesz~\cite{RI-liouville-riemann-integrals-potentials}. Roughly speaking this principle states that any solution of an elliptic equation that vanishes in an open set must be identically zero. It has several applications in inverse problems, control theory and existence theory for PDEs. There are various methods to prove the unique continuation principles for elliptic problems. One could recall Holmgren's uniqueness theorem to obtain the UCP for elliptic PDEs involving real analytic coefficients, see \cite{John:1982}. There are also certain inequalities, such as the doubling inequalities, three sphere inequalities, frequency function methods and Carleman estimates, that can be used to establish the UCP for elliptic equations for less regular coeffiecients, see e.g.~\cite{Lerner:2019}.
However, the method of Carleman estimates has great importance in proving the UCP as well as solving several inverse problems (see e.g.~\cite{DSFKSU09,Is07,KSU07,SU87}).
In this article, we study similar properties for the fractional Laplacians and fractional $p$\,-biharmonic operators. We have proved the following UCP result for the fractional Laplacian in all Bessel potential spaces, excluding the end point $p=\infty$:

\begin{theorem}[UCP]
\label{UCP}
    Let $1\leq p<\infty$, $s\in \R_+\setminus\N$ and $r\in\R$. If $u\in H^{r,p}(\R^n)$ satisfies $(-\Delta)^su=u=0$ in a nonempty open set $V$, then $u\equiv 0$ in $\R^n$.
\end{theorem}

This settles an open problem in \cite[Question 7.1]{CMR20} and extends the result \cite[Theorem~1.2 and Corollary~3.5]{CMR20} from $1 \leq p \leq 2$ to the missing cases $2 < p < \infty$. The proof strategy is similar to \cite[Theorem~1.2]{GSU20}. The higher order cases are proved by an iteration argument with the local operators $(-\Delta)^k$, $k\in\N$, as in \cite[Theorem~1.2]{CMR20}. In particular, the proof uses the Carleman estimates of R\"uland \cite[Proposition~2.2]{Ru15} for the Caffarelli--Silvestre (CS) extension \cite{CS-nonlinera-equations-fractional-laplacians,CS-extension-problem-fractional-laplacian}. However, we need to use additional $L^p$ estimates and make a specific localization argument for the extension problem. The proof of Theorem~\ref{UCP} is presented in Section \ref{sec: UCP}.

Unique continuation properties for the fractional Laplacian and related nonlocal operators have been extensively studied in recent years. We summarize some of these results next. There are strong unique continuation results for $0<s<1$ when one assumes higher regularity of the function \cite{FF-unique-continuation-fractional-ellliptic-equations, Ru15}. In the strong UCP, one replaces the condition $u|_V=0$ by the requirement that $u$ vanishes to infinite order at some point $x_0\in V$. The higher order case $s\in\R^+\setminus \Z$, $s>1$, has been studied recently by several authors \cite{FF-unique-continuation-higher-laplacian, GR-fractional-laplacian-strong-unique-continuation, YA-higher-order-laplacian}. These results however assume some special conditions on the function~$u$, i.e.~they require that $u$ is in a $L^2$ Sobolev space which depends on the power~$s$ of the fractional Laplacian $\fraclaplace$. We also point the interested reader to the work \cite{KR-all-functions-are-s-harmonic} where the author proves a higher order Runge approximation property by $s$-harmonic functions $u$ in the unit ball $B_1$ when $s\in\R^+\setminus\Z$. In the range $0<s<1$, this result has already been established in \cite{DSV-all-functions-are-s-harmonic}. Similar higher regularity approximation results are proved in the article \cite{GSU20} for the fractional Schr\"odinger equation. 
The UCP when $p=2$ and the closely related Runge approximation have been applied in several nonlocal inverse problems to obtain uniqueness results (see e.g. \cite{CRZ22,GRSU-fractional-calderon-single-measurement,GSU20,RZ2022unboundedFracCald}). Another interesting application of the UCP comes from computed tomography \cite{CMR20, IM-unique-continuation-riesz-potential}
as the Riesz potentials (i.e.~the inverses of fractional Laplacians) naturally appear after the so called backprojections in different tomographies.

We denote by $\mathbb{S}_+^m$ the class of functions $A\in L^{\infty}(\R^n;\R^{m\times m})$ taking values in the set of symmetric, positive definite matrices and satisfying the uniform ellipticity condition
    \begin{equation}
    \label{eq: ellipticity intro}
        \lambda^2|v|^2\leq \langle Av,v\rangle\leq \Lambda^2 |v|^2\quad\text{a.e. in}\quad \R^n
    \end{equation}
    for all $v\in\R^m$ and a pair of real numbers $0<\lambda<\Lambda$. The anisotropic fractional $p$\,-biharmonic operator $(-\Delta)^s_{p,A}$ is given weakly by
    \[
        \langle (-\Delta)^s_{p,A} u,v\rangle=\int_{\R^n}|A^{1/2}(-\Delta)^{s/2}u|^{p-2}A(-\Delta)^{s/2}u\cdot (-\Delta)^{s/2}v\,dx
    \]
    for all $u,v\in H^{s,p}(\R^n;\R^m)$ and maps $H^{s,p}(\R^n;\R^m)$ to $(H^{s,p}(\R^n;\R^m))^*$. Using Theorem \ref{UCP}, we are able to prove the following UCP result for the anisotropic fractional $p$\,-biharmonic systems:
\begin{theorem}[UCP for the anisotropic fractional $p$\,-biharmonic operator]
\label{thm: UCP}
    Let $m\in\N$, $1 < p<\infty$, $s> 0$ with $s\notin 2\N$ and $A\in \mathbb{S}_+^m$. Assume that $\Omega\subset\R^n$ is an open set, let $u_1,u_2\in H^{s,p}(\R^n;\R^m)$ and define the functions $v_i\in L^{p'}(\R^n;\R^m)$ by
    \[
        v_i\vcentcolon = |A^{1/2}(-\Delta)^{s/2}u_i|^{p-2}A(-\Delta)^{s/2}u_i
    \]
    for $i=1,2$. If there holds
    \[
        (-\Delta)^s_{p,A} u_1=(-\Delta)^s_{p,A} u_2\quad\text{and}\quad v_1=v_2\quad \text{in}\quad\Omega,
    \]
    then $u_1\equiv u_2$ in $\R^n$.
\end{theorem}
Theorem \ref{thm: UCP} is proved in Section \ref{sec: UCP}. See also Corollary \ref{cor: special cases UCP} for some simpler special cases, and Proposition \ref{prop: measurable UCP} for a measurable UCP with some additional restrictions on the paramaters. We use Theorem \ref{thm: UCP} to show uniqueness in our inverse problems.

\subsection{The exterior data inverse problem and monotonicity relations} 

Ghosh, Salo and Uhlmann showed in \cite{GSU20} that partial exterior DN data associated with the fractional Schrödinger equation of order $0 < s < 1$
\begin{equation}
\begin{split}\label{eq:fracSchördinger}
    ((-\Delta)^s+q)u &= 0\quad \text{in}\quad \Omega, \\
    u&=f\quad \text{in}\quad \Omega_e
\end{split}
\end{equation}
determines uniquely the potential $q \in L^\infty(\Omega)$. The typical solution to the inverse problem is based on the Runge approximation property for the forward model, which follows from the unique continuation principle of the fractional Laplacian and a nonconstructive Hahn--Banach argument. One may determine the potential $q$ from a single measurement \cite{GRSU-fractional-calderon-single-measurement,R-singular-measurement} and the inverse problem is exponentially instable \cite{RS-fractional-calderon-low-regularity-stability,RS17d}. Generalizations of the model problem \eqref{eq:fracSchördinger} have been studied extensively in the literature in the elliptic cases \cite{Covi:2021,covietal2021calderon-directionally-antilocal,CMRU20-higher-order-fracCald,ghosh2021calderon,LL-fractional-semilinear-problems,RZ2022unboundedFracCald,RS-fractional-calderon-low-regularity-stability} and the inverse problem is known to be uniquely solvable for local perturbations of any fixed nonlocal operator with the UCP whenever the forward problem is well-posed \cite{RZ2022unboundedFracCald}. There is also a comprehensive literature considering inverse problems for time-dependent equations with nonlocality, these examples include time-fractional, space-fractional and spacetime-fractional equations \cite{Banerjee:Krishnan:Senapati:2022, Helin:Lassas:Ylinen:Zhang:2020,Kian:Liu:Yamamoto:2022,Kow:Lin:Wang:2022, Lai:Lin:Ruland:2020}.
Inverse problems for nonlocal operators such as the fractional conductivity equation, fractional powers of elliptic operators and fractional spectral Laplacians have been recently studied in \cite{CRZ22,feizmohammadiEtAl2021fractional,ghosh2021calderon, RZ2022counterexamples}. 
More references can be found from the surveys \cite{Sal17, Yamamoto:2022}
and the aforementioned works.

Let $m\in\N$, $1< p<\infty$, $s> 0$ and $A\in \mathbb{S}_+^m$. If $\Omega\subset\R^n$ is an open bounded set and $f\in H^{s,p}(\R^n;\R^m)$, then by Theorem \ref{thm: Homogeneous fractional p-Laplace equation} there exists a unique weak solution $u\in H^{s,p}(\R^n;\R^m)$ to the exterior value problem
    \begin{equation}
    \label{eq: PDE exterior condition intro}
        \begin{split}
            (-\Delta)^s_{p,A}u&=0,\quad \text{in}\quad \Omega,\\
            u&=f,\quad \text{in}\quad \Omega_e.
        \end{split}
    \end{equation}
We define the so called abstract trace space as the quotient $X_p=\nicefrac{H^{s,p}(\R^n;\R^m)}{\widetilde{H}^{s,p}(\Omega;\R^m)}$. We then define the exterior DN map $\Lambda_{p,A}\colon X_p\to X_p^*$ associated with \eqref{eq: PDE exterior condition intro} by 
    \[
        \langle\Lambda_{p,A}(f),g\rangle=\mathcal{A}_{p,A}(u_f,g)
    \]
 for all $f,g\in X_p$, where $u_f$ is the unique weak solution to the homogeneous fractional $p$\,-biharmonic system \eqref{eq: PDE exterior condition intro} and $\mathcal{A}_{p,A}\colon H^{s,p}(\R^n;\R^m)\times H^{s,p}(\R^n;\R^m)\to \R$ is defined as
    \[
        \mathcal{A}_{p,A}(u,v)=\int_{\R^n}|A^{1/2}(-\Delta)^{s/2}u|^{p-2}A(-\Delta)^{s/2}u\cdot (-\Delta)^{s/2}v\,dx
    \]
for all $u,v\in H^{s,p}(\R^n;\R^m)$. More details are given in Sections~\ref{sec: existence theory} and \ref{subsec: trace space and DN maps}. Given a conformal factor $\sigma \in L^\infty(\R^n)$ with $\sigma(x) \geq \sigma_0 > 0$ and a fixed anisotropy $A\in \mathbb{S}_+^m$, we shortly write $\Lambda_{\sigma} = \Lambda_{p,\sigma^{2/p}A}$ (cf. Section~\ref{sec: inverse problems main}). Our main theorem on the related inverse problem is the following single measurement result:

\begin{theorem}\label{main_theorem}
Let $1< p<\infty$ and $s>0$ with $s\notin 2\N$. Suppose that $W\subset \Omega_e$ and $D \subset \R^n$ are given nonempty open sets. Let $\sigma_1, \sigma_2 \in L^{\infty}(\R^n)$ satisfy $\sigma_1(x),\sigma_2(x)\geq \sigma_0>0$ and $\sigma_1\geq \sigma_2$ in $\R^n$. Moreover, suppose that $\sigma_1$ is lower semicontinuous and $\sigma_2$ upper semicontinuous in $D$. If $\Lambda_{\sigma_1}u_0|_W = \Lambda_{\sigma_2}u_0|_W$ holds for some nonzero $u_0\in C_c^{\infty}(W;\R^m)$, then $\sigma_1 = \sigma_2$ in $D \setminus W$.
\end{theorem}

We get a global uniqueness result for classes of conductivities which are assumed to be nontrivial in the whole Euclidean space $\R^n$. In the linear case, without any monotonicity assumptions, the first corresponding result with infinitely many measurements was obtained very recently in \cite{CRZ22} by Covi and the two last named authors. Theorem \ref{main_theorem} directly implies the \emph{global} uniqueness result (which uses two measurements):

\begin{theorem}
\label{main_theorem 2}
    Let $1< p<\infty$ and $s>0$ with $s\notin 2\N$. Suppose that $W\subset \Omega_e$ is a nonempty open set. Let $\sigma_1, \sigma_2 \in L^{\infty}(\R^n)$ satisfy $\sigma_1(x),\sigma_2(x)\geq \sigma_0>0$ and $\sigma_1\geq \sigma_2$ in $\R^n$. Moreover, suppose that $\sigma_1$ is lower semicontinuous and $\sigma_2$ upper semicontinuous in $\R^n$. If $\Lambda_{\sigma_1}f|_W=\Lambda_{\sigma_2}f|_W$ for all $f\in C_c^{\infty}(W;\R^m)$, then $\sigma_1=\sigma_2$ in $\R^n$.
\end{theorem}

We remark our main theorem related to this inverse problem assumes the global monotonicity relation $\sigma_1\geq \sigma_2$ in $\R^n$. Similar limitations are also present in the known uniqueness results for the $p$\,-Calderón problem, which can be thought as a practically relevant local, fully nonlinear, model problem sharing many similarities with the nonlocal problem studied in our article (see Section \ref{sec: pCalderon-intro} for details). On the other hand, many variants of the fractional Calderón problems for linear equations have very strong uniqueness results and the framework of \cite{GSU20} has been very robust to solve many modified problems. Inverse problems for the fractional $p$\,-biharmonic systems require further studies and it remains a partly open question whether nonlocality permits stronger results also for fully nonlinear nonlocal equations. The proof of Theorem \ref{main_theorem} is given in Section \ref{sec: inverse problems main} and it relies on the UCP (Theorems \ref{UCP} and \ref{thm: UCP}) and adapts different methods appearing in the studies of fractional Calderón problems and the classical $p$\,-Calderón problem.

Monotonicity methods have been applied earlier in the fractional Calderón problem for the linear equation \eqref{eq:fracSchördinger}. In particular, Harrach and Lin showed in \cite{Harrach:Lin:2019, Harrach:Lin:2020} 
that $q_1 \leq q_2$ if and only if $\Lambda_{q_1} \leq \Lambda_{q_2}$. Very recently, Lin considered semilinear equations and used monotonicity arguments in the studies of the Calderón problem for nonlinear perturbations of the fractional Laplacians \cite{Lin:2020}.
See also \cite{Harrach:Ullrich:2013,Tamburrino:Rubinacci:2002} for other accounts of the monotonicity methods in inverse problems.

\subsubsection{Further motivation and comparison with the $p$\,-Calderón problem}\label{sec: pCalderon-intro}
 For $1<p<\infty$, consider the Dirichlet problem for the anisotropic $p$\,-Laplace equation
\begin{equation}  \label{p-lapPN0}
\begin{split}
\Div(\sigma\abs{A\nabla u\cdot \nabla u}^{(p-2)/2} A\nabla u) &= 0 \quad \text{in}\quad \Omega, \\
u &= f \quad \text{on}\quad \partial\Omega,
\end{split}
\end{equation}
where $A\in\mathbb{S}_+^n$ and $\sigma \in L^\infty(\Omega)$ with $\sigma \geq \sigma_0 > 0$. The solution of \eqref{p-lapPN0} is the unique minimizer of the $p$\,-Dirichlet energy 
\[
E_p(v) = \int_{\Omega}\sigma|A\nabla v\cdot \nabla v|^{p/2}\,dx 
\]
over all $v\in W^{1,p}(\Omega)$ with $v-f \in W^{1,p}_0(\Omega)$, see \cite{Heinonen:Kilpelainen:Martio:1993,Salo:Zhong:2012}.
Now, let $\mathcal{X}_p$ be the abstract trace space, i.e.~ $\mathcal{X}_p \vcentcolon = W^{1,p}(\Omega)/W_0^{1,p}(\Omega)$.
Then the related DN map $\Lambda^p_{\sigma}\colon \mathcal{X}_p\to\mathcal{X}_p^*$ is weakly defined by
\begin{equation} \label{dnmap_weak_definition}
\langle\Lambda^p_{\sigma}f, g\rangle = \int_{\Omega} \sigma \abs{A\nabla u_f\cdot \nabla u_f}^{(p-2)/2} A\nabla u_f \cdot \nabla v_g \,dx
\end{equation}
for all $f,g\in \mathcal{X}_p$, where $u_f \in W^{1,p}(\Omega)$ is the unique solution of \eqref{p-lapPN0} and $g=v_g|_{\partial\Omega}$ with $v_g\in W^{1,p}(\Omega)$. The $p$\,-Laplace equation is useful in studying certain nonlinear phenomena appearing in nonlinear dielectrics, plastic moulding, nonlinear fluids including electro-rheological and thermo-rheological fluids, fluids governed by a power law, viscous flows in glaciology, or plasticity (see e.g.~\cite{Brander:Kar:Salo:2014} and the references therein). The $n$-Laplace equation has also a connection to the conformal geometry~\cite{Liimatainen:Salo:2012}.

The inverse problem corresponding to the anisotropic $p$\,-Laplace equation is called the \emph{$p$\,-Calderón problem} and asks to recover the conductivity $\sigma$ from the DN map $\Lambda_{\sigma}^p$. This fully nonlinear variant of the Calder\'on problem was introduced by Salo and Zhong in \cite{Salo:Zhong:2012}, where they proved the boundary uniqueness result stating that $\Lambda_{\sigma}^p$ determines $\sigma|_{\partial \Omega}$. First order boundary uniqueness was proved by Brander \cite{Brander:2014}. Other results include inclusion detection and inverse problems in the presence of obstacles \cite{Brander:Kar:Salo:2014,Kar-Wang}.
Numerical studies and linearization approaches were implemented in \cite{Hannukainen-Nuutti-Lauri}.

In \cite{Guo-Kar-Salo}, Guo, Salo and the first named author showed that if the two conductivities $\sigma_1$ and $\sigma_2$ are monotonic in the sense that $\sigma_1\geq \sigma_2$ in $\Omega$ and if $A \in W^{1,\infty}(\Omega; \mathbb{R}^{n \times n})$ has values in $\mathbb{S}^n_+$, then the DN map is injective for Lipschitz conductivites when $n = 2$ for $1<p<\infty$. When $n \geq 3$, similar uniqueness results hold under the assumption that one of the conductivities must be close to a constant and a $C^{1,\alpha}$ regular matrix $A$ is close to identity matrix. Further references on the monotonicity methods include \cite{Brander-Harrach-Kar-Salo,Corbo-Antonio}. Interior uniqueness for the $p$\,-Calderón problem is still open without monotonicity assumptions, which is one motivation to consider nonlocal analogues of this problem. The proof in \cite{Guo-Kar-Salo} is based on the UCP and a monotonicity inequality for the DN maps (a nonlocal version of this inequality is proved in Lemma~\ref{needed_lemma}). 

The UCP of the $p$\,-Laplace equation in three and higher dimensions is an open problem to the best of our knowledge (see \cite[Theorem 2.7]{Seppo:Niko:2014} for a partial result). In two dimensions, the UCP is fairly well understood, see the works of Alessandrini \cite{Alessandrini:1987}, Bojarski--Iwaniec \cite{bi84} and Manfredi \cite{Manfredi:1988}. In the variable coefficient case, see \cite[Proposition 3.3]{Alessandrini:Sigalotti:2001} and \cite{Guo-Kar}. Due to the lack of the UCP in three and higher dimensional domains for the equation \eqref{p-lapPN0}, the interior uniqueness result for the higher dimensional $p$\,-Calderón problem in \cite{Guo-Kar-Salo} has the mentioned, additional, limitations. However, the UCP for our fractional $p$-biharmonic systems (Theorem \ref{thm: UCP}) holds in any dimension and is suitable for the analysis of the related inverse problem. Interestingly, the analogous results for our nonlocal problem (Theorems \ref{main_theorem} and \ref{main_theorem 2}) hold in any dimension without making any additional stronger assumptions.

\section{Preliminaries}
\label{sec: preliminaries}

In this section we first introduce the relevant function spaces used throughout this article and recall the mapping properties of the fractional Laplacians. Finally, we state the (fractional) Poincar\'e inequality on Bessel potential spaces and the Rellich--Kondrachov theorem which will be essential to prove existence (and uniqueness) of solutions to the variational problems and nonlocal, nonlinear, partial differential equations (PDEs) studied in this work.

\subsection{Bessel potential spaces and fractional Laplacians}
\label{subsec: Bessel potentials and co}
 Throughout the article $n,m\in\N$ are fixed natural numbers specifying the dimension of the domain and range of the functions under consideration. We denote the space of Schwartz functions by $\schwartz(\R^n)$ and its dual, the space of tempered distributions, by $\tempered(\R^n)$. We define the Fourier transform on $\schwartz(\R^n)$ by
\[
    \fourier u(\xi)\vcentcolon = \hat u(\xi) \vcentcolon = \int_{\R^n} u(x)e^{-ix \cdot \xi} \,dx
\]
and extend it by duality to $\tempered(\R^n)$. The Fourier transform $\fourier$ acts as an isomorphism on the spaces $\schwartz(\R^n)$, $\tempered(\R^n)$ and we denote its inverse by $\ifourier u$ or $\check{u}$. The Bessel potential of order $s \in \R$ is the Fourier multiplier $\vev{D}^s\colon \tempered(\R^n) \to \tempered(\R^n)$, that is
\begin{equation}\label{eq: Bessel pot}
    \vev{D}^s u \vcentcolon = \ifourier(\vev{\xi}^s\widehat{u}),
\end{equation} 
where $\vev{\xi}\vcentcolon = (1+|\xi|^2)^{1/2}$ is the so-called Japanese bracket. If $s \in \R$ and $1 \leq p < \infty$, the Bessel potential space $H^{s,p}(\R^n)$ is given by
\begin{equation}
\label{eq: Bessel pot spaces}
    H^{s,p}(\R^n) \vcentcolon = \{ u \in \tempered(\R^n)\,;\, \vev{D}^su \in L^p(\R^n)\},
\end{equation}
 endowed with the norm 
 \[
    \norm{u}_{H^{s,p}(\R^n)} \vcentcolon = \norm{\vev{D}^su}_{L^p(\R^n)}.
\]
 For any open set $\Omega\subset \R^n$ and closed set $F\subset\R^n$, we introduce the following local Bessel potential spaces:
\begin{equation}
\label{eq: local bessel pot spaces}
\begin{split}
    \widetilde{H}^{s,p}(\Omega) &\vcentcolon = \mbox{closure of } C_c^\infty(\Omega;\R^m) \mbox{ in } H^{s,p}(\R^n),\\
    H^{s,p}_F(\R^n)&\vcentcolon =\{\,u \in H^{s,p}(\R^n)\,;\, \supp(u) \subset F\,\}.
\end{split}
\end{equation}

If $u\in\tempered(\R^n)$ is a tempered distribution and $s\geq 0$, the fractional Laplacian of order $s$ of $u$ is the Fourier multiplier
\[
    (-\Delta)^su\vcentcolon = \ifourier(|\xi|^{2s}\widehat{u}),
\]
whenever the right hand side is well-defined. If $p\geq 1$ and $t\in\R$, the fractional Laplacian is a bounded linear operator $(-\Delta)^{s}\colon H^{t,p}(\R^n) \to H^{t-2s,p}(\R^n)$. 

Moreover, we denote by $H^{s,p}(\R^n;\R^m),\,\widetilde{H}^{s,p}(\Omega;\R^m),\,H^{s,p}_F(\R^n;\R^m)$ the $m-$fold cartesian product of the above scalar valued spaces and they are naturally endowed with the norm 
\[
    \norm{u}_{H^{s,p}(\R^n;\R^m)} \vcentcolon = \norm{\vev{D}^su}_{L^p(\R^n;\R^m)}.
\]
We extend the Bessel potential operator $\langle D\rangle^s$ and the fractional Laplacian $(-\Delta)^s$ to these spaces by acting componentwise. Clearly, these operators share the same mapping properties on these vectorial spaces as in the scalar valued setting.

\subsection{Poincar\'e inequalities on Bessel potential spaces and the Rellich--Kondrachov theorem}

In this subsection we state a fractional Poincar\'e inequality and a variant of the Rellich--Kondrachov theorem which are adapted to our functional setting. The first result directly follows from Lemma~5.4 in \cite{RZ2022unboundedFracCald}. The second one can be proved, as is done below, using compact embeddings in Besov-type and Triebel--Lizorkin-type spaces on smooth bounded domains \cite{GHS20-Compact-embeddings-Besov-Triebel}. 

\begin{theorem}[{Fractional Poincaré inequality on bounded sets}]
\label{thm: Poincare Bounded sets} 
    Let $\Omega \subset \R^n$ be a bounded open set, $s>0$, $1 < p < \infty$ and $\K=\C$ or $\K=\R^m$. Then there exists $C(n,p,s,\Omega,\K)>0$ such that
\begin{equation}
    \norm{u}_{L^p(\R^n;\K)} \leq C\norm{(-\Delta)^{s/2}u}_{L^p(\R^n;\K)}
\end{equation}
for all $u \in \widetilde{H}^{s,p}(\Omega;\K)$.
\end{theorem}

\begin{theorem}[Rellich--Kondrachov theorem]
\label{thm: Rellich Kondrachov} 
    Let $\Omega \subset \R^n$ be a bounded open set, $s>0$, $1 < p < \infty$ and $\K=\C$ or $\K=\R^m$. Then the embedding $\widetilde{H}^{s,p}(\Omega;\K) \hookrightarrow L^p(\R^n;\K)$ is compact.
\end{theorem}

\begin{proof} 
    Without loss of generality we can restrict ourselves to the complex valued case. Let $\overline{\Omega} \subset \Omega'$ where $\Omega'$ is a smooth bounded domain. By \cite[Remark 2.6, Definition 2.10]{GHS20-Compact-embeddings-Besov-Triebel} one has $F^{s,0}_{p,q}(\Omega')=F^{s}_{p,q}(\Omega)$ for $s\in\R,1<p<\infty$, $0<q\leq\infty$, where $F^{s,\tau}_{p,q}(\Omega)$, $0\leq \tau\leq \infty$, denotes the generalized Triebel-Lizorkin space. Since the Triebel--Lizorkin spaces coincide with the Bessel potential space for $q=2$ we have the identification $F^{s,0}_{p,2}(\Omega')=H^{s,p}(\Omega')$ for $s\in\R,1<p<\infty$. Therefore, \cite[Corollary 3.5]{GHS20-Compact-embeddings-Besov-Triebel} shows that $H^{s,p}(\Omega') \hookrightarrow L^p(\Omega')$ is compact. By the embeddings $\widetilde{H}^{s,p}(\Omega) \hookrightarrow \widetilde{H}^{s,p}(\Omega')\hookrightarrow H^{s,p}(\Omega')$ and $u=0$ a.e. in $\R^n\setminus\overline{\Omega}$ for all $u\in\widetilde{H}^{s,p}(\Omega)$ with $s\geq 0$, it follows that $\widetilde{H}^{s,p}(\Omega)\hookrightarrow L^p(\R^n)$ is compact.
\end{proof}

\subsection{Caffarelli--Silvestre extension problems}
\label{subsec: CS extension problems}

The purpose of this section is to recall the extension technique introduced by Caffarelli and Silvestre in \cite{CS-extension-problem-fractional-laplacian}. More precisely, they showed in their celebrated work that the fractional Laplacian of a smooth bounded function $f\colon \R^n\to\R$ can be obtained as a weighted normal derivative of a function $u\colon \R^{n+1}_+\to \R$ solving a degenerate elliptic equation in $\R^{n+1}_+=\R^n\times (0,\infty)$. 

To make the presentation more transpartent we first fix some notation. We will always use the variable $x$ to label points in $\R^n$, the variable $y$ for points in $\R_+$ and capital letters $X$ when we refer to points in $\R^{n+1}$. Moreover, to highlight that a partial differential operator (PDO) $P=P(\partial)$ acts on $\R^{n+1}_+$ we will use the symbol $\overline{P}$. In particular, we write $\Nabla$, $\DivH$ and $\Lap$ to denote the gradient, the diveregence and the Laplacian on $\R^{n+1}_+$. Of particular interest, related to extension problems, is the following PDO
\begin{equation}
\label{eq: PDO extension 1}
    \DivH(y^{1-2s}\Nabla u(x,y)),
\end{equation}
when $0<s<1$. A straight forward computation shows the identity
\begin{equation}
\label{eq: PDO extension 2}
    y^{-(1-2s)}\DivH(y^{1-2s}\Nabla u(x,y))=\Lap_su(x,y)\quad\text{with}\quad \Lap_s\vcentcolon =\Lap+\frac{1-2s}{y}\partial_y.
\end{equation}
Next we recall the notion of Muckenhoupt weights and introduce a particular class of weighted Sobolev spaces (cf.~\cite{WeightedSobolev}). For any $1<p<\infty$, we say that a weight $w\colon \R^n\to [0,\infty)$ belongs to the Muckenhoupt class $A_p$ if there holds
\[
    \left(\frac{1}{|B|}\int_B w\,dx\right)\left(\frac{1}{|B|}\int_B w^{-p'/p}\,dx\right)^{p/p'}\leq C<\infty
\]
for all balls $B\subset \R^n$, where $1<p'<\infty$ satisfies $1/p+1/p'=1$. A direct calculation shows that $|y|^{1-2s}\,dxdy$ is an $A_2$ weight in $\R^{n+1}$. By \cite[Proposition 7.1.5]{Grafakos1} we deduce that $|y|^{1-2s}$ is an $A_p$ weight for $p\geq 2$. Therefore, following \cite{WeightedSobolev} we can define for any $0<s<1$, $2\leq p<\infty$ and open sets $\Omega\subset \R^{n+1}_+$ the weighted Sobolev spaces $W^{1,p}(\Omega,y^{1-2s})$ as the set of all measurable functions $u\colon \Omega\to \R$ satisfying
\[
    \|u\|_{W^{1,p}(\Omega,y^{1-2s})}\vcentcolon = \|u\|_{L^p(\Omega,y^{1-2s})}+\|\Nabla u\|_{L^p(\Omega,y^{1-2s})} 
\]
with
\[
    \|u\|_{L^p(\Omega,y^{1-2s})}\vcentcolon =\left(\int_{\Omega}|u|^py^{1-2s}\,dX\right)^{1/p}.
\]
As shown in \cite{WeightedSobolev} the spaces $W^{1,p}(\Omega,y^{1-2s})$ endowed with $\|\cdot\|_{W^{1,p}(\Omega,y^{1-2s})}$ are Banach spaces. Moreover, we say that $u\in W^{1,p}_{loc}(\overline{\R^{n+1}_+},y^{1-2s})$ if there holds $u\in W^{1,p}(B_r\times (0,r),y^{1-2s})$ for any $r>0$, where $B_r$ denotes the open ball at the origin with radius $r>0$ in $\R^n$. As usual for $p=2$ we set $W^{1,2}(\Omega,y^{1-2s})=H^1(\Omega,y^{1-2s})$ and $W^{1,2}_{loc}(\overline{\R^{n+1}_+},y^{1-2s})=H^1_{loc}(\R^{n+1}_+,y^{1-2s})$.

Now we are ready to state the aforementioned result of Caffarelli and Silvestre:
\begin{theorem}
\label{thm: CS extension}
    Let $0<s<1$. Then for any $u\in H^s(\R^n)$ there is a unique function $U\in H^1_{loc}(\overline{\R^{n+1}_+},y^{1-2s})$, which solves the extension problem
    \begin{align}\label{p-Laplace}
    \begin{array}{rl} 
    \Lap_sU &\!\!\!= 0 \ \text{ in } \R^{n+1}_+, \\
    U &\!\!\!= u \ \text{ on } \R^n\times \{0\}, \end{array} 
\end{align}
and there exists a constant $c_{n,s}>0$ such that there holds
\[
    (-\Delta)^su(x)=-c_{n,s}\lim_{y\to 0}y^{1-2s}\partial_y U(x,y)
\] 
in $H^{-s}(\R^n)$. Moreover, the unique extension $U$ can be represented as the convolution $U=C_{n,s} P(\cdot,y)\ast u$, where 
\begin{equation}
    P(x,y)=\frac{y^{2s}}{(|x|^2+y^2)^{\frac{n+2s}{2}}}
\end{equation}
is the so called generalized Poisson kernel and $C_{n,s}\vcentcolon = \|P(\cdot,1)\|_{L^1(\R^n)}^{-1}$.
\end{theorem}

Later on in Section~\ref{sec: UCP} we will use this explicit representation of the extension $U$ via the generalized Poisson kernel $P$ to show that the CS extension can be extended to the $L^p$ setting when $p>2$.

\subsection*{Conventions} Throughout the whole article we denote by $1<p'<\infty$ the H\"older conjugated exponent to $1<p<\infty$. Moreover, the dimension $n$ of the domain is fixed to be any natural number but since the results are independent of $n$ we do not further specify it. Furthermore, we denote by $B_r(x_0)$ the ball of radius $r>0$ around $x_0\in\R^n$ in $\R^n$ and by $B^{n+1}(X_0)$ around $X_0\in\R^{n+1}$ in $\R^{n+1}$, we set $B_r\vcentcolon = B_r(0)$, $B_r^{n+1}\vcentcolon = B_r^{n+1}(0)$ and $B_{r,+}^{n+1}\vcentcolon = B_r^{n+1}\cap \overline{\R^{n+1}_+}$.

\section{A variational characterization of the fractional Poincar\'e constant on Bessel potential spaces}
\label{sec: varioational char}

In this section, we show that the fractional $p$\,-biharmonic operator, whose related inverse problem is studied later on, naturally appears when one wants to obtain a variational characterization of the fractional Poincar\'e constant in Theorem~\ref{thm: Poincare Bounded sets}.

\begin{proof}[Proof of Theorem \ref{Theorem: variational characterization of fractional poincare constant}]
    \ref{item 1 var char} This is immediate from the definition of the optimal Poincar\'e constant.
    
    \ref{item 2 var char} Using the fractional Poincar\'e inequality (Theorem~\ref{thm: Poincare Bounded sets} ) and the splitting of the Bessel norm $\|u\|_{H^{s,p}(\R^n)}\sim \|u\|_{L^p(\R^n)}+\|(-\Delta)^{s/2}u\|_{L^p(\R^n)}$ we can endow $\widetilde{H}^{s,p}(\Omega)$ with the equivalent norm $\|u\|_{\widetilde{H}^{s,p}(\Omega)}\vcentcolon =\|(-\Delta)^{s/2}u\|_{L^p(\R^n)}$ for $u\in\widetilde{H}^{s,p}(\Omega)$. Then $(\widetilde{H}^{s,p}(\Omega),\|\cdot\|_{\widetilde{H}^{s,p}(\Omega)})$ is clearly a reflexive Banach space as a closed subspace of a reflexive space. Next we show that $\mathcal{M}_p\subset \widetilde{H}^{s,p}(\Omega)$ is weakly closed in $\widetilde{H}^{s,p}(\Omega)$. Assume $(u_n)_{n\in\N}\subset \mathcal{M}_p$ converges weakly to $u\in\widetilde{H}^{s,p}(\Omega)$. By the Rellich-Kondrachov theorem (Theorem~\ref{thm: Rellich Kondrachov}), the embedding $\widetilde{H}^{s,p}(\Omega)\hookrightarrow L^p(\R^n)$ is compact and thus $u_n\to u$ strongly in $L^p(\R^n)$, but this guarantees that $u\in \mathcal{M}_p$. By the very definition of $\mathcal{E}_p$, it is a coercive and sequentially lower semi-continuous functional on $\mathcal{M}_p$ and hence by \cite[Theorem 1.2]{Variational-Methods} there exists a minimizer $ u\in\mathcal{M}_p$ of $\mathcal{E}_p$ such that $\mathcal{E}_p(u)>0$. The strict positivity follows from the fact that $\mathcal{E}_p(u)=0$ would imply by the fractional Poincar\'e inequality that $u=0$ and so $u$ could not belong to $\mathcal{M}_p$.
    
    \ref{item 3 var char} Fix $\phi\in C_c^{\infty}(\Omega)$ and let $|\epsilon|\leq \epsilon_0$, where $\epsilon_0>0$ is chosen in such a way that $\epsilon_0\|\phi\|_{L^p(\R^n)}\leq 1/2$. Note that this guarantees by the triangle inequality
    \[
        \|u+\epsilon\phi\|_{L^p(\R^n)}\geq \|u\|_{L^p(\R^n)}-|\epsilon|\|\phi\|_{L^p(\R^n)}\geq 1-\epsilon_0\|\phi\|_{L^p(\R^n)}\geq 1/2
    \]
    for all $|\epsilon|\leq \epsilon_0$. Hence, we have $u_{\epsilon}=\frac{u+\epsilon\phi}{\|u+\epsilon\phi\|_{L^p(\R^n)}}\in\mathcal{M}_p$. Next note that 
    \[
        \left.\frac{d}{d\epsilon}\right|_{\epsilon=0}|z+\epsilon w|^p=p|z|^{p-2}(z_1w_1+z_2w_2)=p|z|^{p-2}\text{Re}(\Bar{z}w)
    \]
    for all $z=z_1+iz_2,w=w_1+iw_2\in\C$, since $p>1$. Thus, using the dominated convergence theorem and the fact that $u$ is a minimizer, we obtain
    \[
    \begin{split}
        0&=\left.\frac{d}{d\epsilon}\right|_{\epsilon=0}\mathcal{E}_p(u_{\epsilon})=\left.\frac{d}{d\epsilon}\right|_{\epsilon=0}\frac{\int_{\R^n}|(-\Delta)^{s/2}u+\epsilon(-\Delta)^{s/2}\phi|^p\,dx}{\int_{\R^n}|u+\epsilon\phi|^p\,dx}\\
        &=p\left(\int_{\R^n}|(-\Delta)^{s/2}u|^{p-2}\text{Re}(\overline{(-\Delta)^{s/2}u}(-\Delta)^{s/2}\phi)\,dx-\lambda_{1,s,p}\int_{\R^n}|u|^{p-2}\text{Re}(\overline{u}\phi)\,dx\right)
    \end{split}
    \]
for all $\phi\in C_c^{\infty}(\Omega)$, where we have used $\|u\|_{L^p(\R^n)}=1$ and set $\lambda_{1,s,p}=\|(-\Delta)^{s/2}u\|_{L^p(\R^n)}^p$. Since the fractional Laplacian is generated by a real-valued, radial multiplier we have $\overline{(-\Delta)^{s/2}u}=(-\Delta)^{s/2}\Bar{u}$ for all $u\in H^{s,p}(\R^n)$, and thus we deduce
\[
    \begin{split}
        \int_{\R^n}|(-\Delta)^{s/2}u|^{p-2}((-\Delta)^{s/2}\overline{u}(-\Delta)^{s/2}\phi+(-\Delta)^{s/2}u(-\Delta)^{s/2}\overline{\phi})\,dx&=\lambda_{1,s,p}\int_{\R^n}|u|^{p-2}(\overline{u}\phi+u\overline{\phi})\,dx,\\
        \int_{\R^n}|(-\Delta)^{s/2}u|^{p-2}((-\Delta)^{s/2}u(-\Delta)^{s/2}\phi+(-\Delta)^{s/2}\overline{u}(-\Delta)^{s/2}\overline{\phi})\,dx&=\lambda_{1,s,p}\int_{\R^n}|u|^{p-2}(u\phi+\overline{u}\overline{\phi})\,dx.
    \end{split}
\]
The second identity follows from the first one by replacing $\phi$ by $\Bar{\phi}$. Adding, subtracting these two identities, respectively, we have
\[
\begin{split}
        &\int_{\R^n}|(-\Delta)^{s/2}u|^{p-2}((-\Delta)^{s/2}u+(-\Delta)^{s/2}\overline{u})((-\Delta)^{s/2}\phi+(-\Delta)^{s/2}\overline{\phi})\,dx\\
        &=\lambda_{1,s,p}\int_{\R^n}|u|^{p-2}(u+\overline{u})(\phi+\overline{\phi})\,dx,\\
        &\int_{\R^n}|(-\Delta)^{s/2}u|^{p-2}((-\Delta)^{s/2}u-(-\Delta)^{s/2}\overline{u})((-\Delta)^{s/2}\phi-(-\Delta)^{s/2}\overline{\phi})\,dx\\
        &=\lambda_{1,s,p}\int_{\R^n}|u|^{p-2}(u-\overline{u})(\phi-\overline{\phi})\,dx.
    \end{split}
\]
Choosing $\phi$ real-valued, purely imaginary valued in the first and second equation, respectively, we get
\[
\begin{split}
  \text{Re}\left(\int_{\R^n}\left(|(-\Delta)^{s/2}u|^{p-2}(-\Delta)^{s/2}u (-\Delta)^{s/2}\phi-\lambda_{1,s,p}|u|^{p-2}u\phi\right)\,dx\right)=0\\
  \text{Im}\left(\int_{\R^n}\left(|(-\Delta)^{s/2}u|^{p-2}(-\Delta)^{s/2}u (-\Delta)^{s/2}\phi-\lambda_{1,s,p}|u|^{p-2}u\phi\right)\,dx\right)=0\\
 \end{split}
\]
for all $\phi\in C_c^{\infty}(\Omega;\R)$ and therefore there holds
\[
    \int_{\R^n}|(-\Delta)^{s/2}u|^{p-2}(-\Delta)^{s/2}u (-\Delta)^{s/2}\phi\,dx=\lambda_{1,s,p}\int_{\R^n}|u|^{p-2}u\phi\,dx
\]
for all $\phi\in C_c^{\infty}(\Omega)$. Hence, we have established \eqref{eq: EL eq} for all $\phi\in C_c^{\infty}(\Omega)$. Next, we show that it in fact holds for all $\phi\in\widetilde{H}^{s,p}(\Omega)$. If $u\in H^{s,p}(\R^n)$ then we deduce from H\"older's inequality that
\[
    U\vcentcolon =|(-\Delta)^{s/2}u|^{p-2}(-\Delta)^{s/2}u\in L^{p'}(\R^n)\quad \text{with}\quad \|U\|_{L^{p'}(\R^n)}=\|(-\Delta)^{s/2}u\|_{L^p(\R^n)}^{p-1}<\infty,
\]
where $1<p'<\infty$ satisfies $1/p+1/p'=1$.
Using the mapping properties of the fractional Laplacian and H\"older's inequality, we see that the derived identity holds for all $\phi\in \widetilde{H}^{s,p}(\Omega)$.

\ref{item 4 var char} Since $v\in\widetilde{H}^{s,p}(\Omega)$, we can test \eqref{eq: identity 4} by $\Bar{v}$ to obtain
    \[
    \int_{\R^n}|(-\Delta)^{s/2}v|^p\,dx=\mu\int_{\R^n}|v|^p\,dx
    \]
    and hence $\mu\in\R_+$. Therefore $w=v/\|v\|_{L^p(\R^n)}\in \mathcal{M}_p$ satisfies
    \[
        \mu=\mathcal{E}_p(w)\geq \mathcal{E}_p(u)=\lambda_{1,s,p}
    \]
    for any minimizer $u\in\mathcal{M}_p$.
\end{proof}

\begin{remark}
    The applied methods can be adapted to construct minimizers of energy functionals which has an additional term involving a weighted $L^q$ norm of $u$ similarly as in  \cite{ClassPBiharm} or even more general situations.
\end{remark}

\section{Existence theory for fractional $p$\,-biharmonic type equations}
\label{sec: existence theory}

In Section~\ref{subsec: The anisotropic fractional $p$-biharmonic operator}, we first introduce a class of anisotropic fractional $p$\,-biharmonic operators, which naturally arise in the Euler--Lagrange equations of certain energy functionals. In Section~\ref{subsec: Existence of solutions}, we then prove well-posedness results for these anisotropic fractional $p$\,-biharmonic operators in the cases of pure interior source and pure exterior value as discussed in the introduction. 

\subsection{Anisotropic fractional $p$\,-biharmonic operators}
\label{subsec: The anisotropic fractional $p$-biharmonic operator}

We first define a class of matrices which will be used throughout this article and introduce the associated energy functionals. Then we introduce the anisotropic fractional $p$\,-biharmonic operators and make several remarks in which we explain briefly the used terminology and discuss other possibilities of defining anisotropic fractional $p$\,-biharmonic operators.

\begin{definition}[Anisotropic $p$\,-energies]
    \label{def: assumptions energy functional}
    Let $m\in\N$, $s>0$ and $1<p<\infty$. We denote by $\mathbb{S}_+^m$ the class of functions $A\in L^{\infty}(\R^n;\R^{m\times m})$ taking values in the set of symmetric, positive definite matrices and satisfying the ellipticity condition
    \begin{equation}
    \label{eq: ellipticity}
        \lambda^2|v|^2\leq \langle Av,v\rangle\leq \Lambda^2 |v|^2\quad\text{a.e. in}\quad \R^n
    \end{equation}
    for all $v\in\R^m$ and a pair of real numbers $0<\lambda<\Lambda$. For all $A\in \mathbb{S}_+^m$, we define the related anisotropic $p$\,-energy by
    \[
        \mathcal{E}_{p,A}\colon H^{s,p}(\R^n;\R^m)\to \R,\quad \mathcal{E}_{p,A}(u)=\frac{1}{p}\int_{\R^n}|A^{1/2}(-\Delta)^{s/2}u|^p\,dx
    \]
    for all $u\in H^{s,p}(\R^n;\R^m)$, where $A^{1/2}$ is the unique square root of $A$.
\end{definition}

\begin{proposition}[Anisotropic fractional $p$\,-biharmonic operators]
\label{prop: definition of biharmonic op}
    Let $m\in\N$, $1<p<\infty$, $s>0$ and $A\in \mathbb{S}_+^m$ with ellipticity constants $0<\lambda<\Lambda$. Then the anisotropic fractional $p$\,-biharmonic operator $(-\Delta)^s_{p,A}$ is given by
    \[
        \langle (-\Delta)^s_{p,A} u,v\rangle=\int_{\R^n}|A^{1/2}(-\Delta)^{s/2}u|^{p-2}A(-\Delta)^{s/2}u\cdot (-\Delta)^{s/2}v\,dx
    \]
    for all $u,v\in H^{s,p}(\R^n;\R^m)$ and maps $H^{s,p}(\R^n;\R^m)$ to $(H^{s,p}(\R^n;\R^m))^*$. Moreover, there holds
    \begin{equation}
    \label{eq: estimate fractional biharm op}
        \|(-\Delta)^{s}_{p,A}u\|_{(H^{s,p}(\R^n;\R^m))^*}\leq C_0\Lambda^p\|(-\Delta)^{s/2}u\|_{L^p(\R^n)}^{p-1}
    \end{equation}
    for all $u\in H^{s,p}(\R^n;\R^m)$ for some $C_0>0$.
\end{proposition}

\begin{remark}
    If $m=1$ and $A=1$, then we set $(-\Delta)^s_p\vcentcolon =(-\Delta)^s_{p,1}$ and call it fractional $p$\,-biharmonic operator. This terminology is motivated by the fact that the classical $p$\,-biharmonic operator is given by
    \[
        (-\Delta)^2_p\vcentcolon = \Delta(|\Delta u|^{p-2}\Delta u),
    \]  
    which is a nonlinear variant of the usual biharmonic operator $\Delta^2$, and this operator coincides with $(-\Delta)^s_p$ for $s=2$. 
\end{remark}

\begin{remark}
Here we want to highlight that one could define other variants of anisotropic fractional $p$\,-biharmonic operators solely based on the fractional Laplacian and a coefficient field $A$ as:
\begin{enumerate}[(i)]
    \item\label{case 1} $A(-\Delta)^s_{p,m} u$ or $(-\Delta)^{s}_{p,m} A u$, where $(-\Delta)_{p,m}^s\vcentcolon =(-\Delta)_{p,\mathbf{1}_m}^s$.
    \item\label{case 2} or $(-\Delta)^{s/2}\left(|(-\Delta)^{s/2}u|^{p-2}A(-\Delta)^{s/2}u\right)$
\end{enumerate}
As long as $A\in\mathbb{S}_+^m$ is sufficiently smooth the behaviour of solutions to the associated boundary value problem of the first two alternatives are quite similar as for the usual fractional $p$\,-biharmonic operator and thus we think there do not arise new interesting phenomena. On the other hand in the case \ref{case 2}, we observe
\[\begin{split}
&\int_{\R^n}|(-\Delta)^{s/2}u|^{p-2}A(-\Delta)^{s/2}u\cdot (-\Delta)^{s/2}v\,dx \\
&=
\int_{\R^n}|(-\Delta)^{s/2}u|^{p-2}(-\Delta)^{s/2}u\cdot ([A^T,(-\Delta)^{s/2}]v+(-\Delta)^{s/2}(A^T v)\,dx
\end{split}
\]
for all $u\in H^{s,p}(\R^n;\R^m)$, $v\in C_c^{\infty}(\Omega;\R^m)$.
If $[A^T,(-\Delta)^{s/2}]=0$ then the solution to the associated boundary value problem can again be easily obtained, but if the commutator is nonzero then this definition could still lead to interesting, nontrivial solutions.
\end{remark}

\begin{proof}
    Let $u\in H^{s,p}(\R^n;\R^m)$ and note that the assumptions on $A$ guarantee the estimate $|A^{1/2}w|\leq \Lambda|w|$ for all $w\in\R^m$. Then the Cauchy--Schwartz inequality, H\"older's inequality and the mapping properties of the fractional Laplacian imply
    \[
    \begin{split}
    |\langle (-\Delta)^s_{p,A} u,v\rangle|&=\left|\int_{\R^n}|A^{1/2}(-\Delta)^{s/2}u|^{p-2}A^{1/2}(-\Delta)^{s/2}u\cdot A^{1/2}(-\Delta)^{s/2}v\,dx\right|\\
    &\leq \||A^{1/2}(-\Delta)^{s/2}u|^{p-1}\|_{L^{p'}(\R^;\R^m)}\|A^{1/2}(-\Delta)^{s/2}v\|_{L^p(\R^n;\R^m)}\\
    &\leq \Lambda^{p}\|(-\Delta)^{s/2}u\|_{L^p(\R^n;\R^m)}^{p-1}\|(-\Delta)^{s/2}v\|_{L^p(\R^n;\R^m)}\\
    &\leq C_0\Lambda^{p}\|(-\Delta)^{s/2}u\|_{L^p(\R^n;\R^m)}^{p-1}\|v\|_{H^{s,p}(\R^n;\R^m)}
    \end{split}
    \]
    for all $v\in H^{s,p}(\R^n;\R^m)$. Taking the supremum over all nonzero $v\in H^{s,p}(\R^n;\R^m)$, we obtain the estimate \eqref{eq: estimate fractional biharm op}. The rest of the statement now follows from the mapping properties of the fractional Laplacian and we can conclude the proof.
\end{proof}

\subsection{Well-posedness results for anisotropic fractional $p$\,-biharmonic operators}
\label{subsec: Existence of solutions}

Next we introduce the used notion of weak solutions and show two closely related well-posedness results for the anisotropic fractional $p$\,-biharmonic operators. The second well-posedness result will then be used later to define the DN maps related to the exterior value problems for these operators.

\begin{definition}[Weak solutions]
\label{def: weak solutions}
     Let $m\in\N$, $1< p<\infty$, $s> 0$ and $A\in \mathbb{S}_+^m$. Suppose that $f\in H^{s,p}(\R^n;\R^m)$ and $F\in (\widetilde{H}^{s,p}(\Omega;\R^m))^*$. Then we say that $u\in H^{s,p}(\R^n;\R^m)$ is a weak solution to the exterior value problem 
     \begin{equation}
    \label{eq: weak solutions}
        \begin{split}
            (-\Delta)^s_{p,A}u&=F,\quad \text{in}\quad \Omega\\
            u&=f,\quad \text{in}\quad \Omega_e,
        \end{split}
    \end{equation}
    if there holds 
    \[
        \langle (-\Delta)^s_{p,A}u,v\rangle=\langle F,v\rangle\quad\text{and}\quad u-f\in \widetilde{H}^{s,p}(\Omega;\R^m)
    \]
    for all $v\in \widetilde{H}^{s,p}(\Omega;\R^m)$.
\end{definition}

Before proceeding we recall the following well-known estimates:

\begin{lemma}[{cf.~\cite[eq.~(2.2)]{Simon}, \cite[Lemma 5.1-5.2]{Estimate_p_Laplacian}}]
\label{lemma: Auxiliary lemma}
    Let $m\in\N$, $1<p<\infty$, then there exists $c_p>0$ such that for all $x,y\in\R^{m}$ there holds
    \[
        ( |x|^{p-2}x-|y|^{p-2}y)\cdot (x-y)\geq c_p|x-y|^p
    \]
    if $p\geq 2$ and
    \[
        (|x|^{p-2}x-|y|^{p-2}y)\cdot (x-y)\geq c_p\frac{|x-y|^2}{(|x|+|y|)^{2-p}}
    \]
    if $1<p<2$.
\end{lemma}

\begin{theorem}[Inhomogeneous equations with zero Dirichlet condition]
\label{theorem:Inhomogeneous fractional p-Laplace equation with zero exterior condition}
    Let $m\in\N$, $1< p<\infty$, $s> 0$ and $A\in \mathbb{S}_+^m$. If $\Omega\subset\R^n$ is an open bounded set and $F\in (\widetilde{H}^{s,p}(\Omega;\R^m))^*$, then there exists a unique weak solution $u\in\widetilde{H}^{s,p}(\Omega;\R^m)$ of
    \begin{equation}
    \label{eq: inhomogeneous p fractional Laplace equation}
        \begin{split}
            (-\Delta)^s_{p,A}u&=F,\quad \text{in}\quad \Omega,\\
            u&=0,\quad\, \text{in}\quad \Omega_e.
        \end{split}
    \end{equation}
    Moreover, the solution $u\in\widetilde{H}^{s,p}(\Omega;\R^m)$ satisfies
    \begin{equation}
    \label{eq: estimate solution}
        \|u\|_{H^{s,p}(\R^n;\R^m)}\leq C\|F\|_{(\widetilde{H}^{s,p}(\Omega;\R^m))^*}^{1/(p-1)}
    \end{equation}
    for some $C>0$.
\end{theorem}
\begin{proof}
    Assume that $A\in\mathbb{S}_+^m$ satisfies \eqref{eq: ellipticity} with ellipticity constants $0<\lambda<\Lambda$. Using the fractional Poincar\'e inequality (Theorem~\ref{thm: Poincare Bounded sets}) and the splitting 
    \[
        \|u\|_{H^{s,p}(\R^n;\R^m)}\sim \|u\|_{L^p(\R^n;\R^m)}+\|(-\Delta)^{s/2}u\|_{L^p(\R^n;\R^m)}\quad\text{for all}\quad u\in H^{s,p}(\R^n;\R^m),
    \]
    we can endow $\widetilde{H}^{s,p}(\Omega;\R^m)$ with the equivalent norm 
    \[
      \|u\|_{\widetilde{H}^{s,p}(\Omega;\R^m)}\vcentcolon =\|(-\Delta)^{s/2}u\|_{L^p(\R^n;\R^m)}.
    \]
     Then $\widetilde{H}^{s,p}(\Omega;\R^m)$ with the norm $\|\cdot\|_{\widetilde{H}^{s,p}(\Omega;\R^m)}$ is a reflexive Banach space. More precisely, this follows from the fact that $\widetilde{H}^{s,p}(\Omega;\R^m)$ with $\|\cdot\|_{H^{s,p}(\R^n;\R^m)}$ is a reflexive Banach spaces as a closed subspace of a reflexive Banach space, but the first one is isomorphic to $\widetilde{H}^{s,p}(\Omega;\R^m)$ endowed with $\|\cdot\|_{\widetilde{H}^{s,p}(\Omega;\R^m)}$ and so the latter is itself a reflexive Banach space. Next, we define
    \[
        \mathcal{E}_{p,A,F}(u)=\mathcal{E}_{p,A}(u)-\langle F,u\rangle
    \]
    for all $u\in\widetilde{H}^{s,p}(\Omega;\R^m)$, where $\mathcal{E}_{p,A}$ is the anisotropic $p$\,-energy from Definition~\ref{def: assumptions energy functional}. By assumption we have $|A^{1/2}w|\geq \lambda|w|$ for all $w\in\R^m$, and hence Young's inequality implies
    \[
    \begin{split}
        |\mathcal{E}_{p,A,F}(u)|&\geq \frac{\lambda^p}{p}\|u\|_{\widetilde{H}^{s,p}(\Omega;\R^m)}^p-\|F\|_{(\widetilde{H}^{s,p}(\Omega;\R^m))^*}\|u\|_{\widetilde{H}^{s,p}(\Omega;\R^m)}\\
        &\geq \left(\lambda^p/p -\epsilon\right)\|u\|_{\widetilde{H}^{s,p}(\Omega;\R^m)}^p-C_{\epsilon}\|F\|_{(\widetilde{H}^{s,p}(\Omega;\R^m))^*}^{p'}
    \end{split}
    \]
    for all $\epsilon>0$ and $C_{\epsilon}=(\epsilon p)^{-p'/p}/p'$. Hence by choosing $\epsilon=\lambda^p/(2p)$, we obtain
    \[
        \|u\|_{\widetilde{H}^{s,p}(\Omega;\R^m)}^p\leq C\left(|\mathcal{E}_{p,A,F}(u)|+\|F\|_{(\widetilde{H}^{s,p}(\Omega;\R^m))^*}^{p'}\right)
    \]
    for some $C>0$ and therefore $\mathcal{E}_{p,A,F}$ is coercive, in the sense that 
    \[
        \mathcal{E}_{p,A,F}(u)\to\infty\quad\text{if}\quad \|u\|_{\widetilde{H}^{s,p}(\Omega;\R^m)}\to\infty.
    \]
    Next note that $\mathcal{E}_{p,A}$ is a convex, continuous functional on $\widetilde{H}^{s,p}(\Omega;\R^m)$ and hence using \cite[Proposition~2.10]{ConvexityOptimization}, we deduce that $\mathcal{E}_{p,A,F}$ is weakly lower semi-continuous on $\widetilde{H}^{s,p}(\Omega;\R^m)$. 
     Therefore, by \cite[Theorem 1.2]{Variational-Methods} there is a minimizer $u\in\widetilde{H}^{s,p}(\Omega;\R^m)$ of $\mathcal{E}_{p,A,F}$ for all $1<p<\infty$. 
     
     By H\"older's inequality and the dominated convergence theorem, we see that $\mathcal{E}_{p,A,F}$ is a $C^1-$functional for all $1<p<\infty$. Let $u_{\epsilon}=u+\epsilon\phi$ with $\phi\in C_c^{\infty}(\Omega;\R^m)$ and $\epsilon\in\R$. Since $u$ is a minimizer of $\mathcal{E}_{p,A,F}$, the function $\epsilon\mapsto \mathcal{E}_{p,A,F}(u_{\epsilon})$ attains its minimum at $\epsilon=0$. Therefore, by H\"older's inequality and the dominated convergence theorem, we obtain
    \[
    \begin{split}
        0&=\left.\frac{d}{d\epsilon}\right|_{\epsilon=0}\mathcal{E}_{p,A,F}(u_{\epsilon})=\int_{\R^n}|A^{1/2}(-\Delta)^{s/2}u|^{p-2}A^{1/2}(-\Delta)^{s/2}u\cdot A^{1/2}(-\Delta)^{s/2}\phi\,dx-\langle F,\phi\rangle\\
        &=\int_{\R^n}|A^{1/2}(-\Delta)^{s/2}u|^{p-2}A(-\Delta)^{s/2}u\cdot (-\Delta)^{s/2}\phi\,dx-\langle F,\phi\rangle.
    \end{split}
    \]
    By approximation, we deduce that the minimizer $u$ solves \eqref{eq: inhomogeneous p fractional Laplace equation} as asserted. For the uniqueness statement, we distinguish the two cases $2\leq p<\infty$ and $1<p<2$:
    
    \begin{enumerate}[(i)]
    \item First assume that $2\leq p<\infty$. By applying Lemma~\ref{lemma: Auxiliary lemma} to the vectors $x=u_s,\,y=v_s$, where $u_s=A^{1/2}(-\Delta)^{s/2}u$ and $v_s=A^{1/2}(-\Delta)^{s/2}v$ with $u,v\in H^{s,p}(\R^n;\R^m)$, we obtain the following strong monotonicity property
        \begin{equation}
        \label{eq: Strict monotonicity property}
        \begin{split}
            &\int_{\R^n}( |A^{1/2}(-\Delta)^{s/2}u|^{p-2}A(-\Delta)^{s/2}u
            -|A^{1/2}(-\Delta)^{s/2}v|^{p-2}A(-\Delta)^{s/2}v)\\
            &\quad\cdot ((-\Delta)^{s/2}u-(-\Delta)^{s/2}v)\,dx\\
            &\geq \lambda c_p\|(-\Delta)^{s/2}u-(-\Delta)^{s/2}v\|^p_{L^p(\R^n;\R^m)}
        \end{split}
        \end{equation}
        for all $u,v\in H^{s,p}(\R^n;\R^m)$. If $u,v\in \widetilde{H}^{s,p}(\Omega;\R^m)$ then $w=u-v\in\widetilde{H}^{s,p}(\Omega;\R^m)$ and hence if they solve \eqref{eq: inhomogeneous p fractional Laplace equation} the left hand side of \eqref{eq: Strict monotonicity property} is zero and therefore $\|u-v\|_{\widetilde{H}^{s,p}(\Omega;\R^m)}=0$, which in turn implies $u=v$ in $\widetilde{H}^{s,p}(\Omega;\R^m)$. Thus, the constructed minimizer is the unique solution.
        \\
    \item Next let $1<p<2$. We apply the second identity in Lemma~\ref{lemma: Auxiliary lemma} to $x=u_s,y=v_s$, where $u_s=A^{1/2}(-\Delta)^{s/2}u$, $v_s=A^{1/2}(-\Delta)^{s/2}v$ with $u,v\in H^{s,p}(\R^n;\R^m)$, raise it to the power $p/2$, integrate over $\R^n$ and use H\"older's inequality to obtain
    \[
    \begin{split}
        &c_p^{p/2}\int_{\R^n}|u_s-v_s|^p\,dx\leq\int_{\R^n}(( |u_s|^{p-2}u_s-|v_s|^{p-2}v_s)\cdot(u_s-v_s))^{p/2}(|u_s|+|v_s|)^{(2-p)p/2}\,dx\\
        &\leq \|( |u_s|^{p-2}u_s-|v_s|^{p-2}v_s)\cdot(u_s-v_s)\|_{L^1(\R^n)}^{p/2}\||u_s|+|v_s|\|_{L^p(\R^n)}^{p(2-p)/2}\\
        &\leq \left(\int_{\R^n}( |u_s|^{p-2}u_s-|v_s|^{p-2}v_s)\cdot(u_s-v_s)\,dx\right)^{p/2}(\|u_s\|_{L^p(\R^n;\R^m)}+\|v_s\|_{L^p(\R^n;\R^m)})^{p(2-p)/2}
    \end{split}
    \]
    where we used $\frac{2-p}{2}+\frac{p}{2}=1$ and
    \[
        \begin{split}
            ( |u_s|^{p-2}u_s-|v_s|^{p-2}v_s)\cdot (u_s-v_s)^{p/2}\in L^{2/p}(\R^n),\quad (|u_s|+|v_s|)^{(2-p)p/2}\in L^{2/(2-p)}(\R^n).
        \end{split}
    \]
    Hence, by the ellipticity condition on $A$ there holds
    \begin{equation}
    \label{eq: monotonicity subquadratic case}
        \begin{split}
            &\|u-v\|_{\widetilde{H}^{s,p}(\Omega;\R^m)}\\
            &\leq \frac{\Lambda^{1-p/2}}{c_p^{1/2}\lambda}(\|u\|_{\widetilde{H}^{s,p}(\Omega;\R^m)}+\|v\|_{\widetilde{H}^{s,p}(\Omega;\R^m)})^{1-p/2}\\ 
            &\quad\cdot\left(\int_{\R^n}( |A^{1/2}(-\Delta)^{s/2}u|^{p-2}A(-\Delta)^{s/2}u\right.\\
            &\quad\left.-|A^{1/2}(-\Delta)^{s/2}v|^{p-2}A(-\Delta)^{s/2}v)\cdot((-\Delta)^{s/2}u-(-\Delta)^{s/2}v)\,dx\right)^{1/2}
        \end{split}
    \end{equation}
    for all $u,v\in H^{s,p}(\R^n;\R^m)$. If $u,v\in \widetilde{H}^{s,p}(\Omega;\R^m)$ then $w=u-v\in\widetilde{H}^{s,p}(\Omega;\R^m)$ and hence if they satisfy \eqref{eq: inhomogeneous p fractional Laplace equation} the second term on the right hand side of \eqref{eq: monotonicity subquadratic case} is zero and hence $\|u-v\|_{\widetilde{H}^{s,p}(\Omega;\R^m)}=0$, which in turn implies $u=v$ in $\widetilde{H}^{s,p}(\Omega;\R^m)$. Therefore, the constructed minimizer is the unique solution. 
    \end{enumerate}
    
    Estimate \eqref{eq: estimate solution} follows directly by testing \eqref{eq: inhomogeneous p fractional Laplace equation} with $u\in\widetilde{H}^{s,p}(\Omega)$ and using $A\in\mathbb{S}_+^m$. In fact, by Poincar\'e's inequality we have
    \[
    \begin{split}
        \lambda^p\|(-\Delta)^{s/2}u\|_{L^p(\R^n;\R^m)}^p&\leq \int_{\R^n}|A^{1/2}(-\Delta)^{s/2}u|^{p-2}A(-\Delta)^{s/2}u\cdot(-\Delta)^{s/2}u\,dx\\
        &=\langle F,u\rangle  \leq \|F\|_{(\widetilde{H}^{s,p}(\Omega;\R^m))^*}\|u\|_{H^{s,p}(\R^n;\R^m)}\\
    &\leq C\|F\|_{(\widetilde{H}^{s,p}(\Omega;\R^m))^*}\|(-\Delta)^{s/2}u\|_{L^p(\R^n;\R^m)}.
    \end{split}
    \]
    This shows the estimate \eqref{eq: estimate solution} and we can conclude the proof.
\end{proof}

\begin{theorem}[Homogeneous equations with nonzero Dirichlet condition]
\label{thm: Homogeneous fractional p-Laplace equation}
    Let $m\in\N$, $1< p<\infty$, $s> 0$ and $A\in \mathbb{S}_+^m$. If $\Omega\subset\R^n$ is an open bounded set and $u_0\in H^{s,p}(\R^n;\R^m)$, then there exists a unique weak solution $u\in H^{s,p}(\R^n;\R^m)$ of
    \begin{equation}
    \label{eq: PDE exterior condition}
        \begin{split}
            (-\Delta)^s_{p,A}u&=0,\quad \,\,\,\text{in}\quad \Omega,\\
            u&=u_0,\quad \text{in}\quad \Omega_e.
        \end{split}
    \end{equation}
   Moreover, the unique solution $u\in H^{s,p}(\R^n;\R^m)$ satisfies the estimate
    \begin{equation}
    \label{eq: estimate solution with exterior condition}
        \|(-\Delta)^{s/2}u\|_{L^p(\R^n;\R^m)}\leq (\Lambda/\lambda)^p\|(-\Delta)^{s/2}u_0\|_{L^p(\R^n;\R^m)}.
    \end{equation}
\end{theorem}
\begin{proof}
    Assume that the ellipticity condition for $A$ holds with parameters $0<\lambda<\Lambda<\infty$. Let us define the affine subspace
    \[
        \widetilde{H}^{s,p}_{u_0}(\Omega;\R^m)=\{u\in H^{s,p}(\R^n;\R^m)\colon\, u-u_0\in\widetilde{H}^{s,p}(\Omega;\R^m)\}\subset H^{s,p}(\R^n;\R^m)
    \]
and denote by $\mathcal{E}'_{p,A}$ the restriction of $\mathcal{E}_{p,A}$ to $\widetilde{H}^{s,p}_{u_0}(\Omega;\R^m)$. First observe that $\widetilde{H}^{s,p}_{u_0}(\Omega;\R^m)$ is weakly closed in the reflexive Banach space $H^{s,p}(\R^n;\R^m)$. In fact, if $(u_n)_{n\in\N}\subset \widetilde{H}^{s,p}_{u_0}(\Omega;\R^m)$ converges weakly to $u\in H^{s,p}(\R^n;\R^m)$ in $H^{s,p}(\R^n;\R^m)$, then $u_n-u_0$ converges weakly to $u-u_0\in H^{s,p}(\R^n;\R^m)$ in $H^{s,p}(\R^n;\R^m)$. As weak limits are contained in the weak closure, the weak closure of convex sets coincide with the strong closure and $\widetilde{H}^{s,p}(\Omega;\R^m)$ is closed in $H^{s,p}(\R^n;\R^m)$ we deduce that $u-u_0\in \widetilde{H}^{s,p}(\Omega;\R^m)$. Therefore we have shown that $\widetilde{H}^{s,p}_{u_0}(\Omega;\R^m)$ is weakly closed in $H^{s,p}(\R^n;\R^m)$. 

Using $|A^{1/2}w|\geq \lambda|w|$ for all $w\in\R^m$, the fractional Poincar\'e inequality (Theorem~\ref{thm: Poincare Bounded sets} ), the usual splitting of the Bessel potential norm $\|\cdot\|_{H^{s,p}(\R^n;\R^m)}$ and the convexity of $x\mapsto |x|^p$ we deduce 
\[
\begin{split}
    |\mathcal{E}'_{p,A}(u)|&\geq C\|u\|_{H^{s,p}(\R^n;\R^m)}^p-C'\|u_0\|_{H^{s,p}(\R^n;\R^m)}^p
\end{split}
\]
for all $u\in \widetilde{H}^{s,p}_{u_0}(\Omega;\R^m)$ and hence $\mathcal{E}'_{p,A}$ is coercive on $\widetilde{H}^{s,p}_{u_0}(\Omega;\R^m)$. By the same argument as in the proof of Theorem~\ref{theorem:Inhomogeneous fractional p-Laplace equation with zero exterior condition} the functional $\mathcal{E}'_{p,A}$ is weakly lower semi-continuous on $\widetilde{H}^{s,p}_{u_0}(\Omega;\R^m)$ with respect to the norm of the space $H^{s,p}(\R^n;\R^m)$. Therefore by \cite[Theorem 1.2]{Variational-Methods} there is a minimizer $u\in\widetilde{H}^{s,p}_{u_0}(\Omega;\R^m)$ of $\mathcal{E}'_{p,A}$ for all $1<p<\infty$. Repeating the argument of Theorem~\ref{theorem:Inhomogeneous fractional p-Laplace equation with zero exterior condition} we see that the minimizer is unique. 

To show the estimate \eqref{eq: estimate solution with exterior condition} we proceed similarly as in the proof of Theorem~\ref{theorem:Inhomogeneous fractional p-Laplace equation with zero exterior condition}, namely we test \eqref{eq: PDE exterior condition} by $u-u_0\in\widetilde{H}^{s,p}(\Omega;\R^m)$, use $A\in\mathbb{S}_+^m$ and apply H\"older's inequality to deduce
    \[
    \begin{split}
        \lambda^p\|(-\Delta)^{s/2}u\|_{L^p(\R^n;\R^m)}^p&\leq \int_{\R^n}|A^{1/2}(-\Delta)^{s/2}u|^{p-2}A(-\Delta)^{s/2}u\cdot(-\Delta)^{s/2}u\,dx\\
        &=\int_{\R^n}|A^{1/2}(-\Delta)^{s/2}u|^{p-2}A(-\Delta)^{s/2}u\cdot(-\Delta)^{s/2}(u-u_0)\,dx\\
        &+\int_{\R^n}|A^{1/2}(-\Delta)^{s/2}u|^{p-2}A(-\Delta)^{s/2}u\cdot(-\Delta)^{s/2}u_0\,dx\\
        &=\int_{\R^n}|A^{1/2}(-\Delta)^{s/2}u|^{p-2}A(-\Delta)^{s/2}u\cdot(-\Delta)^{s/2}u_0\,dx\\
        &\leq \Lambda^p\|(-\Delta)^{s/2}u\|^{p-1}_{L^p(\R^n;\R^m)}\|(-\Delta)^{s/2}u_0\|_{L^p(\R^n;\R^m)}
    \end{split}
    \]
    which in turn implies \eqref{eq: estimate solution with exterior condition}.
\end{proof}

\section{Abstract trace space and DN maps for anisotropic fractional $p$\,-biharmonic operator}
\label{subsec: trace space and DN maps}

In this section, we introduce the basic notions needed to set up the inverse problem related to the anisotropic fractional $p$\,-biharmonic operators, namely the abstract trace space and the DN map. 

\begin{definition}
    Let $m\in\N$, $1< p<\infty$, $s> 0$ and $\Omega\subset \R^n$ be an open set. Then we define the abstract trace space as $X_p=\nicefrac{H^{s,p}(\R^n;\R^m)}{\widetilde{H}^{s,p}(\Omega;\R^m)}$ and endow it with the quotient norm
    \[
    \|[f]\|_{X_p}=\inf_{\phi\in\widetilde{H}^{s,p}(\Omega;\R^m)}\|f-\phi\|_{H^{s,p}(\R^n;\R^m)}
    \]
    for all $[f]\in X_p$.
\end{definition}
\begin{remark}
By standard arguments, one can easily show that $X_p$ is a Banach space as long as $\widetilde{H}^{s,p}(\Omega;\R^m)\neq H^{s,p}(\R^n;\R^m)$. To simplify the notation, we will usually denote elements in $X_p$ by $f$ instead of the more precise notation $[f]$.
\end{remark}

In the next lemma, we show that for any $u_0\in X_p$ there is a unique solution $u\in H^{s,p}(\R^n;\R^m)$ of the related pure exterior value problem for any anisotropic fractional $p$\,-biharmonic operator.

\begin{lemma}
\label{lemma: uniqueness of solutions on trace space}
    Let $m\in\N$, $1< p<\infty$, $s> 0$, $A\in \mathbb{S}_+^m$ and assume $\Omega\subset\R^n$ is an open bounded set. If $u^1_0,u_0^2\in H^{s,p}(\R^n;\R^m)$ are such that $u_0^1-u_0^2\in\widetilde{H}^{s,p}(\Omega;\R^m)$ and suppose $u_1,u_2\in H^{s,p}(\R^n;\R^m)$ are the unique solutions of the exterior value problems
    \[
        \begin{split}
            (-\Delta)^s_{p,A}u_1&=0,\quad\,\,\, \text{in}\quad \Omega,\\
            u&=u_0^1,\quad \text{in}\quad \Omega_e,
        \end{split}
    \]
    and 
    \[
        \begin{split}
            (-\Delta)^s_{p,A}u_2&=0,\quad \text{in}\quad\,\,\, \Omega,\\
            u&=u_0^2,\quad \text{in}\quad \Omega_e,
        \end{split}
    \]
    then $u_1\equiv u_2$ in $\R^n$.
\end{lemma}
\begin{proof}
    First of all note that by assumption, we have
    \[
        u_1-u_2=(u_1-u_0^1)-(u_2-u_0^2)+(u_0^1-u_0^2)\in \widetilde{H}^{s,p}(\Omega;\R^m).
    \]
    Therefore using the strong monotonicity property (eq.~\eqref{eq: Strict monotonicity property} and \eqref{eq: monotonicity subquadratic case}) and the fact that $u_1,u_2$ are solutions of $(-\Delta)^s_{p,A}v=0$ in $\Omega$, we deduce $u_1\equiv u_2$ in $\R^n$.
\end{proof}

\begin{lemma}[DN map]
\label{lemma: DN maps}
    Let $m\in\N$, $1< p<\infty$, $s> 0$, $A\in \mathbb{S}_+^m$ and assume $\Omega\subset\R^n$ is an open bounded set. Then the DN map $\Lambda_{p,A}\colon X_p\to X_p^*$ given by
    \[
        \langle\Lambda_{p,A}(f),g\rangle=\mathcal{A}_{p,A}(u_f,g)
    \]
    for $f,g\in X_p$ is well-defined, where $u_f$ is the unique weak solution to the homogeneous fractional $p$\,-biharmonic system with exterior value $f$ and $\mathcal{A}_{p,A}\colon H^{s,p}(\R^n;\R^m)\times H^{s,p}(\R^n;\R^m)\to \R$ is defined as
    \[
        \mathcal{A}_{p,A}(u,v)=\int_{\R^n}|A^{1/2}(-\Delta)^{s/2}u|^{p-2}A(-\Delta)^{s/2}u\cdot (-\Delta)^{s/2}v\,dx
    \]
    for all $u,v\in H^{s,p}(\R^n;\R^m)$. Moreover, there exists $C>0$ such that
    \[
        |\langle\Lambda_{p,A}(f),g\rangle|\leq C\|f\|_{X_p}^{p-1}\|g\|_{X_p}
    \]
    for all $f,g\in X_p$.
\end{lemma}
\begin{proof}
    Let us assume that $A\in\mathbb{S}_+^m$ has ellipticity constants $0<\lambda<\Lambda$. By Lemma~\ref{lemma: uniqueness of solutions on trace space}, we know that $u_f$ is independent of the chosen representative and since $u_f$ is a solution to the homogeneous fractional $p$\,-biharmonic systems, we see that $\mathcal{A}_{p,A}(u_f,g)$ is as well independent of the representative of $g$. We now proceed similarly as in the proof of Proposition~\ref{prop: definition of biharmonic op}. Using H\"older's inequality, $A\in \mathbb{S}_+^m$, Theorem~\ref{thm: Homogeneous fractional p-Laplace equation}, the continuity of the fractional Laplacian, we deduce the estimate
    \[
    \begin{split}
        |\langle\Lambda_{p,A}(f),g\rangle|&\leq \Lambda^{p-1}\|(-\Delta)^{s/2}u_{\widetilde{f}}\|^{p-1}_{L^p(\R^n;\R^m)}\|(-\Delta)^{s/2}\widetilde{g}\|_{L^p(\R^n;\R^m)}\\
        &\leq C\|(-\Delta)^{s/2}\widetilde{f}\|^{p-1}_{L^p(\R^n;\R^m)}\|\widetilde{g}\|_{H^{s,p}(\R^n;\R^m)}\\
        &\leq C\|\widetilde{f}\|_{H^{s,p}(\R^n;\R^m)}^{p-1}\|\widetilde{g}\|_{H^{s,p}(\R^n;\R^m)}
    \end{split}
    \]
    for all $f,g\in X_p$ and all representative $\widetilde{f},\widetilde{g}$ of $f$ and $g$. This in turn implies
    \[
         |\langle\Lambda_{p,A}(f),g\rangle|\leq C\|f\|^{p-1}_{X_p}\|g\|_{X_p}
    \]
    and thus $\Lambda_{p,A}$ is indeed well-defined. 
\end{proof}

\section{Caffarelli--Silvestre extension and unique continuation principles for nonlocal operators in Bessel potential spaces}
\label{sec: UCP}

\subsection{Caffarelli--Silvestre extension in $L^p$ for $p\neq 2$}
\label{sec:caffarelliSilvestre}

We first show a preliminary lemma which deals with elementary properties of the generalized Poisson kernel. The proof of \ref{item 3 Poisson kernel} strongly follows the one of \cite[Proposition~B.1]{GenCKNThm}, where also the estimate in \ref{item 4 Poisson kernel} is stated.

\begin{lemma}[Properties of the Poisson kernel]
\label{lem: properties Poisson kernel}
    Let $0<s<1$ and denote by $P$ the generalized Poisson kernel. Then the following statements hold:
    \begin{enumerate}[(i)]
        \item\label{item 1 Poisson kernel} $P(\cdot,y)\in L^q(\R^n)$ for all $1\leq q<\infty$ and $y>0$ with
        \[
            \|P(\cdot,y)\|^q_{L^q(\R^n)}=\frac{\omega_n}{2}y^{n(1-q)}B\left(\frac{n}{2},\frac{n(q-1)}{2}+sp\right),
        \]
        where $B(x,y)$ denotes the Euler Beta function,
        \item\label{item 2 Poisson kernel} $(C_{n,s}P(\cdot,y))_{y>0}$ is a Dirac sequence,
        \item\label{item 3 Poisson kernel} $P\in C^{\infty}(\R^{n+1}_+)$ solves $\Lap_s P=0$ in $\R^{n+1}_+$,
        \item\label{item 4 Poisson kernel} for all $k\in\N$, there exists $C_k>0$ such that
        \[
            |\partial_y^kP(x,y)|\leq C_k\frac{y^{2s-k}}{(|x|^2+y^2)^{\frac{n+2s}{2}}}
        \]
         for all $(x,y)\in \R^{n+1}_+$ and $\partial_yP$ is radially symmetric in $x\in\R^n$.
    \end{enumerate}
\end{lemma}

\begin{proof}
    \ref{item 1 Poisson kernel} Fix $y>0$ and $1\leq q<\infty$, then using the change of variables $z=x/y$, we obtain
    \[
    \begin{split}
        \int_{\R^n}|P(x,y)|^q\,dx&=y^{n(1-q)}\int_{\R^n}\frac{1}{(|x|^2+1)^{\frac{n+2s}{2}q}}\,dz=\omega_ny^{n(1-q)}\int_0^{\infty}\frac{r^{n-1}}{(1+r^2)^{\frac{n+2s}{2}q}}\,dr\\
        &=\frac{\omega_n}{2} y^{n(1-q)}\int_0^1(1-t)^{\frac{n}{2}(q-1)+sq-1}t^{\frac{n}{2}-1}\,dt\\
        &=\frac{\omega_n}{2} y^{n(1-q)}B\left(\frac{n}{2},\frac{n}{2}(q-1)+sq\right),
    \end{split}
    \]
where in the second equality we used polar coordinates, then made the change of variables $r^2=\frac{t}{1-t}$ with $dr=\frac{1}{2(1-t)^2}\sqrt{\frac{1-t}{t}}dt$ and finally used the product rule for the Beta function, which is given by $B(x,y)=\int_0^1t^{x-1}(1-t)^{y-1}\,dt$ for all $x,y>0$.
    
\ref{item 2 Poisson kernel} For all $(x,y)\in\R^{n+1}_+$, we have $P(x,y)\geq 0$ and by the definition of $C_{n,s}$ there holds $\|C_{n,s}P(\cdot,y)\|_{L^1(\R^n)}=1$. Moreover, by Lebesgue's dominated convergence theorem, we have $\|P(\cdot,y)\|_{L^1(\R^n\setminus B_{\epsilon})}\to 0$ as $y\to 0$ for any $\epsilon>0$ and therefore the claim follows.

\ref{item 3 Poisson kernel} The smoothness of $P$ directly follows from the assumption that $y>0$. Since for $y>0$, there holds
    \[
        \partial_i|X|^{-\gamma}=-\gamma X_i|X|^{-(\gamma+2)}
    \]
    for all $1\leq i\leq n+1$ and $\gamma>0$. Hence, we have
    \begin{equation}
    \label{eq: laplacian}
    \begin{split}
        \Lap |X|^{-\gamma}&=\gamma(\gamma+2)|X|^{-{\gamma+2}}-(n+1)\gamma|X|^{-(\gamma+2)}\\
        &=\gamma(\gamma-(n-1))|X|^{-(\gamma+2)}.
    \end{split}
    \end{equation}
    Therefore, for $\gamma,\beta>0$, we obtain
    \[
    \begin{split}
        &\Lap_s \left(\frac{y^{\beta}}{|X|^{\gamma}}\right)
        =\Lap \left(\frac{y^{\beta}}{|X|^{\gamma}}\right)+\frac{1-2s}{y}\partial_y\left(\frac{y^{\beta}}{|X|^{\gamma}}\right)\\
        &=\gamma(\gamma-(n-1))\frac{y^{\beta}}{|X|^{\gamma+2}}+\beta(\beta-1)\frac{y^{\beta-2}}{|X|^{\gamma}}-2\beta\gamma\frac{y^{\beta}}{|X|^{\gamma+2}}+(1-2s)\beta\frac{y^{\beta-2}}{|X|^{\gamma}}-(1-2s)\gamma\frac{y^{\beta}}{|X|^{\gamma+2}}\\
        &= \beta(\beta-1+(1-2s)\frac{y^{\beta-2}}{|X|^{\gamma}}+\gamma(\gamma-2\beta-(n-1)-(1-2s))\frac{y^{\beta}}{|X|^{\gamma+2}}.
    \end{split}
    \]
    If we take $\beta=2s$, $\gamma=n+2s$, then both coefficients are zero and we see that $\Lap_s P=0$. This shows $\Lap_s P=0$ in $\R^{n+1}_+$.
    
    \ref{item 4 Poisson kernel} This estimate is stated in \cite[Proposition~B.1, eq. (78)]{GenCKNThm}. It can be proved by a direct, but a bit lengthy, computation using the generalized Leibniz rule and the formula of Faà di Bruno or induction.
\end{proof}

\begin{lemma}
\label{prop: integrability of CS extension}
    Let $0<s<1$, $1<p<\infty$, denote by $P$ the generalized Poisson kernel, assume $u\in L^p(\R^n)$ and let $U(\cdot,y)\vcentcolon = C_{n,s}P(\cdot,y)\ast u$. Then 
    \begin{equation}
    \label{eq: integrability CS extension}
        \|U(\cdot,y)\|_{L^p(\R^n)}\leq \|u\|_{L^p(\R^n)},
    \end{equation}
    where $C_{n,s}=\|P(\cdot,1)\|_{L^1(\R^n)}^{-1}$, and $U\in L^p_{loc}(\overline{\R^{n+1}_+},y^{1-2s})$. Moreover, $U$ solves
    \begin{align}
         \begin{array}{rl} 
            \Lap_sU &\!\!\!= 0 \ \text{ in } \R^{n+1}_+, \\
            U &\!\!\!= u \ \text{ on } \R^n\times \{0\}. 
        \end{array} 
    \end{align}
\end{lemma}

\begin{remark}
    The estimate $\|P(\cdot,y)\ast u\|_{L^p(\R^n)}\leq C\|u\|_{L^p(\R^n)}$ also follows from \cite[Theorem~2.1]{StingaTorrea}, but since our proof is less involved we presented the argument here.
\end{remark}

\begin{proof}
    The estimate \eqref{eq: integrability CS extension} follows by Young's inequality and the property \ref{item 1 Poisson kernel} of Lemma~\ref{lem: properties Poisson kernel}. This estimate implies
    \[
    \begin{split}
        &\left(\int_0^R\int_{\R^n}|U|^py^{1-2s}\,dxdy\right)^{1/p}=\left(\int_0^Ry^{1-2s}\|U(\cdot,y)\|_{L^p(\R^n)}^p\,dy\right)^{1/p}\\
        &\leq \|u\|_{L^p(\R^n)}\left(\int_0^Ry^{1-2s}\,dy\right)^{1/p}=C_sR^{2(1-s)/p}\|u\|_{L^p(\R^n)}
    \end{split}
    \]
    for any $R>0$ and some $C_s>0$. Therefore there holds $U\in L^p_{loc}(\overline{\R^{n+1}_+},y^{1-2s})$.
    Next, we verify that $U$ solves $\Lap_s U=0$ in $\R^{n+1}$. By the assertions~\ref{item 1 Poisson kernel} and \ref{item 3 Poisson kernel} of Lemma~\ref{lem: properties Poisson kernel}, we have $\partial_y^kP\in L^q(\R^n)$ for all $1\leq q\leq \infty$. Moreover, it follows by a direct calculation that
    \[
    \begin{split}
        |\Delta P(x,y)|&=y^{1-2s}\left|-n(n+2s)|X|^{-(n+2s+2)}+(n+2s)(n+2s+2)|x|^2|X|^{-(n+2s+4)}\right|\\
        &\leq Cy^{1-2s}|X|^{-(n+2s+2)}\leq Cy^{-(1+4s)}P(x,y)
    \end{split}
    \]  
    which again belongs to $L^q(\R^n)$ whenever $1\leq q\leq \infty$ and $y>0$.
    Therefore, the assertion follows by Young's inequality, the dominated convergence theorem and the property~\ref{item 1 Poisson kernel} of Lemma~\ref{lem: properties Poisson kernel}. Finally, the boundary condition is a consequence of \ref{item 2 Poisson kernel} in Lemma~\ref{lem: properties Poisson kernel}.
\end{proof}

Next we recall the following basic properties of the CS extension from \cite[Section~5.1]{CDV-Bernstein-technique}:

\begin{lemma}
\label{thm: CS extension p>2}
    Let $0<s<1$ and $u\in C^{\infty}(\R^n)\cap W^{2,\infty}(\R^n)$. Then the CS extension $U(\cdot,y)\vcentcolon = C_{n,s}P(\cdot,y)\ast u\in C^{\infty}(\R^{n+1}_+)\cap C(\overline{\R^{n+1}_+})$ is the unique bounded solution of
    \begin{align}
    \label{eq: harm extension}
         \begin{array}{rl} 
            \Lap_sU &\!\!\!= 0 \ \text{ in } \R^{n+1}_+, \\
            U &\!\!\!= u \ \text{ on } \R^n\times \{0\}
        \end{array} 
    \end{align}
    and there holds
    \begin{equation}
    \label{eq: recov frac Lapl}
       -c_{n,s} \lim_{y\to 0}y^{1-2s}\partial_yU=(-\Delta)^su.
    \end{equation}
\end{lemma}

\subsection{Unique continuation principles for the fractional Laplacian}
\label{sec:UCPfractionalLaplacia}
In the study of nonlocal inverse problems, one important property of the fractional Laplacian is the unique continuation principle (UCP) (cf.~\cite[Theorem 1.2]{CMR20}): \\

\emph{Let $r\in \R$, $s\in \R_+\setminus\N$. If $u \in H^{r}(\R^n)$ satisfies $(-\Delta)^su=u=0$ in a nonempty open set $V$, then $u\equiv 0$ in $\R^n$.}\\

Theorem~\ref{UCP} shows that the UCP holds also in the Bessel potential spaces $H^{r,p}(\R^n)$ for $p\neq 2$ and $r\in\R$. For $1\leq p<2$ this is well-known (see \cite[Corollary 3.5]{CMR20}) but for $2<p<\infty$ this result is to the best of our knowledge new. We first prove the following reduction lemma:

\begin{lemma}
\label{lemma: reduction}
    Let $2<p<\infty$, $0<s<1$ and define $W^{\infty,p}(\R^n)=\bigcap_{k\in\N}W^{k,p}(\R^n)$
. Suppose that $(-\Delta)^su=u=0$ in a nonempty open set $V$ and $u\in C_b^{\infty}(\R^n)\cap W^{\infty,p}(\R^n)$ implies that $u\equiv 0$ in $\R^n$. Then Theorem~\ref{UCP} holds true.
\end{lemma}

\begin{proof}
    As already noted, we can assume without loss of generality that $2<p<\infty$ and $s\in\R_+\setminus \N$. We first show that it suffices to prove Theorem~\ref{UCP} for functions $u$ in the class $C^{\infty}_b(\R^n)\cap W^{\infty,p}(\R^n)$. For this purpose let $(\rho_{\epsilon})_{\epsilon>0}$ be a sequence of standard mollifiers and fix $u\in H^{r,p}(\R^n)$ with $r\in\R$ satisfying $(-\Delta)^su=u=0$ in some nonempty open subset $V\subset \R^n$. Since the Bessel potential operator commutes with convolution, we have $u_{\epsilon}\vcentcolon = u\ast \rho_{\epsilon}\in H^{t,p}(\R^n)$ for all $t\in\R$ and thus the Sobolev embedding implies $u_\epsilon \in C^{\infty}_b(\R^n)$. This shows $u_\epsilon \in C^{\infty}_b(\R^n)\cap W^{\infty,p}(\R^n)$. Next fix some precompact open subset $\Omega$ with $\overline{\Omega}\subset V$ and choose $\epsilon_0>0$ such that $\overline{B_{\epsilon_0}(x)}\subset V$ for all $x\in \Omega$. Then we clearly have $\phi\ast \rho_{\epsilon}\in C_c^{\infty}(V)$ for all $\phi\in C_c^{\infty}(\Omega)$, $0<\epsilon<\epsilon_0$, and therefore $(-\Delta)^su_{\epsilon}=u_{\epsilon}=0$ in $\Omega$ as the fractional Laplacian commutes with mollification. Now if this implies $u_{\epsilon}=0$ then the convergence $u_{\epsilon}\to u$ in $H^{r,p}(\R^n)$ shows $u=0$. This shows that it is enough to prove Theorem~\ref{UCP} for functions $u\in C^{\infty}_b(\R^n)\cap W^{\infty,p}(\R^n)$. 
    
    Next, following \cite{CMR20}, we show that Theorem~\ref{UCP} holds. If $0<s<1$, then Theorem~\ref{UCP} holds by assumption of Lemma~\ref{lemma: reduction} and the first part of the proof. Thus we can assume $s>1$. Suppose that $u\in C^{\infty}_b(\R^n)\cap W^{\infty,p}(\R^n)$ satisfies $(-\Delta)^su=u=0$ in some nonempty open set $V\subset \R^n$. We set $t\vcentcolon =s-k\in (0,1)$, where $k\in\N$ is the unique integer such that $k<s<s=k+1$. As in the proof of \cite[Theorem 1.2]{CMR20}, we can see that $(-\Delta)^{k}u\in C^{\infty}_b(\R^n)\cap W^{\infty,p}(\R^n)$ satisfies $(-\Delta)^{t}u=u=0$ in $V$, since $(-\Delta)^k$ is a local operator. Now by assumption there holds $(-\Delta)^ku\equiv 0$ in $\R^n$. Using \cite[Lemma 3.1]{CMR20}, we deduce $u\equiv 0$ in $\R^n$. Therefore, we can conclude the proof.
\end{proof}

\begin{proof}[Proof of Theorem~\ref{UCP}]
    As noted earlier it is sufficient to consider the case $2<p<\infty$. By Lemma~\ref{lemma: reduction}, we can assume without loss of generality that $u\in C^{\infty}_b(\R^n)\cap H^{t,p}(\R^n)$ for any $t\in\R$ and assume that $u=(-\Delta)^su=0$ in some nonempty open set $V\subset\R^n$. Let us denote by $U\in L^{p}_{loc}(\overline{\R^{n+1}_+},y^{1-2s})$ the CS extension of $u$, where the regularity follows from Lemma~\ref{prop: integrability of CS extension}. Then by Lemma~\ref{thm: CS extension p>2}, we know that $U$ solves \eqref{eq: harm extension} and there holds 
    \begin{equation}
        -c_{n,s}\lim_{y\to 0}y^{1-2s}\partial_yU=(-\Delta)^su.
    \end{equation} Next we show that $\Nabla U\in L^2_{loc}(\overline{\R^{n+1}_+},y^{1-2s})$. Since $U\in C^{\infty}_b(\R^n)$ and there holds $(-\Delta)^su\in L^{\infty}(\R^n)$ (cf.~\cite[Lemma~3.2]{DINEPV-hitchhiker-sobolev}), we have $y^{1-2s}\partial_y U\in C(\overline{\R^{n+1}_+})$. Hence, by \cite[Proposition~3.6]{CS-nonlinera-equations-fractional-laplacians}, we deduce $\|y^{1-2s}\partial_yU\|_{L^{\infty}(\R^{n+1}_+)}\leq C$. Therefore, there holds $\partial_y U\in L^{2}_{loc}(\overline{\R^{n+1}_+},y^{1-2s})$. In fact for any $R>0$, $\Omega\Subset\R^n$ we have
    \[
    \int_0^R\int_{\Omega}y^{1-2s}|\partial_yU|^2\,dxdy\leq \|y^{1-2s}\partial_yU\|^2_{L^{\infty}(\R^{n+1}_+)}|\Omega|\int_0^Ry^{-(1-2s)}\,dy<\infty,
    \]
    since $1-2s\in (-1,1)$. By Young's inequality and $\|C_{n,s}P(\cdot,y)\|_{L^1(\R^n)}=1$, we obtain
    \[
        \begin{split}
            \|\nabla U\|^p_{L^p(\R^n\times (0,R),y^{1-2s})}&=\int_0^Ry^{1-2s}\|\nabla U(\cdot,y)\|_{L^p(\R^n)}^p\,dy\\
            &\leq C\int_0^Ry^{1-2s}\|P(\cdot,y)\ast \nabla u\|_{L^p(\R^n)}^p\,dy\\
            &\leq C_{n,s,R}\|\nabla u\|_{L^p(\R^n)}^p<\infty
        \end{split}
    \]
    for any $R>0$. Since $p>2$ and the calculation above, H\"older's inequality implies $\Nabla U\in L^2_{loc}(\overline{\R^{n+1}_+},y^{1-2s})$ and therefore $U\in H^1_{loc}(\R^{n+1}_+,y^{1-2s})\cap L^{\infty}(\R^{n+1}_+)$. 
    
    Next fix $r>0$, $x_0\in\R^n$ such that $B_{2r}(x_0)\subset V$ and let $\eta\in C_c^{\infty}(\overline{\R^{n+1}})$ be a cutoff function supported in $B^{n+1}_{2r,+}(X_0)$ with $\eta|_{B_{r,+}^{n+1}(X_0)}=1$, where $X_0=(x_0,0)$. Now $V\vcentcolon = \eta U\in H^{1}_{loc}(\R^{n+1},y^{1-2s})\cap L^{\infty}(\R^n)$ solves
    \begin{align}
    \label{eq: harm extension cut off}
         \begin{array}{rl} 
            \Lap_sV &\!\!\!= g \ \text{ in } \R^{n+1}_+, \\
            V &\!\!\!= f \ \text{ on } \R^n\times \{0\},
        \end{array} 
    \end{align}
    where $f\vcentcolon =\eta(\cdot,0)u$ and
    $g\vcentcolon = U\Lap \eta+2\Nabla U\cdot\Nabla \eta+\frac{1-2s}{y}U\partial_y \eta$.
    Note that $g$ vanishes in $B_{r,+}^{n+1}(X_0)$. Moreover, by the product rule we deduce $f=(-\Delta)^sf=0$ in $B_{r,+}^{n+1}(X_0)$. Therefore, $V\in H^{1}_{loc}(\R^{n+1},y^{1-2s})\cap L^{\infty}(\R^n)$ solves
    \begin{align}
    \label{eq: harm extension cut off}
         \begin{array}{rl} 
            \Lap_sV &\!\!\!= 0  \text{ in } B^{n+1}_{r}(X_0), \\
            V &\!\!\!= 0 \ \text{ in }  B_{r}(x_0),\\
            \lim\limits_{y\to 0}y^{1-2s}\partial_y V&\!\!\!= 0  \text{ in } B_{r}(x_0).
        \end{array} 
    \end{align}
    Therefore, we can apply \cite[Proposition~2.2]{Ru15} to deduce $V=0$ in $B_{r,+}^{n+1}(X_0)$ and thus $U= 0$ in  $B_{r,+}^{n+1}(X_0)$. Since $U$ solves an elliptic PDE with real analytic coefficients, we deduce from the analytic regularity theory that $U$ is real analytic in $\R^{n+1}_+$ (see~e.g.~\cite[Chapter~8 and 9]{HO:analysis-of-pdos}). Now as $U$ vanishes on an open set we deduce that $U=0$ in $\R^{n+1}_+$ but this implies $u=0$ in $\R^n$.
\end{proof}

\subsection{Unique continuation principles for the anisotropic fractional $p$\,-biharmonic operator}

Similarly as for the uniqueness results for the inverse problems related to the fractional Schr\"odinger equations, an important role is played by the unique continuation properties of the anisotropic fractional $p$\,-biharmonic operator $(-\Delta)^s_{p,A}$. In this section, we show that the UCP for the fractional Laplacians naturally lead to certain variants of unique continuation principles for the operators $(-\Delta)^s_{p,A}$. Under additional monotonicity properties on the coefficient fields $A$, this nonlocal phenomenon allows us to deduce 
uniqueness statements for the related inverse problems of anisotropic fractional $p$\,-biharmonic systems with monotonic classes of coefficients.

\begin{proof}[Proof of Theorem \ref{thm: UCP}]
     We assume throughout the proof that the matrix valued function $A$ satisfies the ellipticity condition with parameters $0<\lambda<\Lambda$. First, consider the case $p\geq 2$ and note that this implies $1<p'\leq 2$ and therefore applying \cite[Corollary 3.5]{CMR20} componentwise, we deduce that there holds $v_1=v_2$ in $\R^n$. By integrating the first identity in Lemma~\ref{lemma: Auxiliary lemma} with $x=(-\Delta)^{s/2}u_1, y=(-\Delta)^{s/2}u_2$ over all of $\R^n$, we obtain the following strong monotonicity property
    \[
    \begin{split}
        0=&\int_{\R^n}(v_1-v_2)((-\Delta)^{s/2}u_1-(-\Delta)^{s/2}u_2)\,dx\\
        =&\int_{\R^n}(|A^{1/2}(-\Delta)^{s/2}u_1|^{p-2}A^{1/2}(-\Delta)^{s/2}u_1-|A^{1/2}(-\Delta)^{s/2}u_2|^{p-2}A^{1/2}(-\Delta)^{s/2}u_2)\\
        &\cdot(A^{1/2}(-\Delta)^{s/2}u_1-A^{1/2}(-\Delta)^{s/2}u_2)\,dx\\
        \geq& c_p\int_{\R^n}|A^{1/2}(-\Delta)^{s/2}u_1-A^{1/2}(-\Delta)^{s/2}u_2|^p\,dx\\
        \geq& c_p\lambda^{p}\int_{\R^n}|(-\Delta)^{s/2}u_1-(-\Delta)^{s/2}u_2|^p\,dx.
    \end{split}
    \]
    If $sp<n$, then the Hardy--Littlewood--Sobolev lemma shows $u_1=u_2$ a.e. in $\R^n$. On the other hand, if $sp\geq n$, then the above calculation ensures $(-\Delta)^{s/2}u_1=(-\Delta)^{s/2}u_2$ in $\R^n$. Now we set $u_i\vcentcolon =u_1^{i}-u_2^{i}\in H^{s,p}(\R^n)$, $i=1,\ldots,m$, and claim that the support of $\ifourier u_i$ is contained in $\{0\}$. In fact, if $\psi\in C_c^{\infty}(\R^n\setminus\{0\})$, then $\phi= |\xi|^{-s}\psi\in C_c^{\infty}(\R^n\setminus \{0\})$, and therefore there holds
    \[
    \begin{split}
        \langle \ifourier u_i,\psi\rangle &= \langle \ifourier u_i, |\xi|^s\phi\rangle =\langle u_i,(-\Delta)^{s/2}\check{\phi}\rangle=\langle (-\Delta)^{s/2}u_i,\check{\phi}\rangle=0,
    \end{split}
    \]
    since $\check{\phi}\in\mathscr{S}_0(\R^n)$. By \cite[Exercise 2.67]{MI:distribution-theory}, we deduce that $\ifourier u_i$ has a unique representation of the form
    \[
        \ifourier u_i=\sum_{|\alpha|\leq N_i}a_{\alpha}^{i}D^{\alpha}_{\xi}\delta_0
    \]
    for some $a_{\alpha}^{i}\in \C$ and $N_i\in\N_0$. Therefore 
    \[
        u_i=\sum_{|\alpha|\leq N_i}a^{i}_{\alpha}x^{\alpha},
    \]
    but since $u_i\in L^p(\R^n)$, we deduce $u_i\equiv 0$ in $\R^n$ for all $1\leq i\leq m$ and therefore $u_1\equiv u_2$ in $\R^n$.
    
    Next consider the case $1<p<2$. Then using Theorem~\ref{UCP}, we deduce $v_1=v_2$ in $\R^n$. Similarly as in the previous case, we derive, by integrating the second identity in Lemma~\ref{lemma: Auxiliary lemma} with $x=(-\Delta)^{s/2}u_1, y=(-\Delta)^{s/2}u_2$ over all of $\R^n$, the estimate
    \[
    \begin{split}
        0=&\int_{\R^n}(v_1-v_2)((-\Delta)^{s/2}u_1-(-\Delta)^{s/2}u_2)\,dx\\
        =&\int_{\R^n}\left(|A^{1/2}(-\Delta)^{s/2}u_1|^{p-2}A^{1/2}(-\Delta)^{s/2}u_1-|A^{1/2}(-\Delta)^{s/2}u_2|^{p-2}A^{1/2}(-\Delta)^{s/2}u_2\right)\\
        &\cdot(A^{1/2}(-\Delta)^{s/2}u_1-A^{1/2}(-\Delta)^{s/2}u_2)\\
        \geq &c_p\int_{\R^n}\frac{|A^{1/2}(-\Delta)^{s/2}u_1-A^{1/2}(-\Delta)^{s/2}u_2|^2}{(|A^{1/2}(-\Delta)^{s/2}u_1|+|A^{1/2}(-\Delta)^{s/2}u_1|)^{2-p}}\,dx\\
        \geq &c_p\frac{\lambda^2}{\Lambda^{p-2}}\int_{\R^n}\frac{|(-\Delta)^{s/2}u_1-(-\Delta)^{s/2}u_2|^2}{(|(-\Delta)^{s/2}u_1|+|(-\Delta)^{s/2}u_1|)^{2-p}}\,dx.
    \end{split}
    \]
    This can only hold if $(-\Delta)^{s/2}u_1=(-\Delta)^{s/2}u_2$ a.e. in $\R^n$. By the same argument as in the case $2\leq p<\infty$, we have $u_1=u_2$ a.e. in $\R^n$. Therefore, we can conclude the proof.
\end{proof}

\begin{corollary}[Special cases]
\label{cor: special cases UCP}
    Let $1< p<\infty$, $s>0$ with $s\notin 2\N$ and $\Omega\subset \R^n$ be an open set. Moreover, assume that Theorem~\ref{UCP} holds.
    \begin{enumerate}[(i)]
        \item If $u\in H^{s,p}(\R^n)$ satisfies
            \[
              (-\Delta)^s_p u=(-\Delta)^{s/2}u=0\quad \text{in}\quad\Omega,
            \]
            then $u\equiv 0$ in $\R^n$.
        \item If $u_1,u_2\in H^{s,p}(\R^n)$ satisfy
            \[
              (-\Delta)^s_p (u_1-u_2)=(-\Delta)^{s/2}(u_1-u_2)=0\quad \text{in}\quad\Omega,
            \]
            then $u_1\equiv u_2$ in $\R^n$.
    \end{enumerate}
\end{corollary}
\begin{proof} The assertions directly follow from Theorem~\ref{thm: UCP}.
\end{proof}

To prove a measurable UCP for anisotropic fractional $p$\,-biharmonic operators, we will need the following estimate:

\begin{lemma}({\cite[Lemma 5.3]{Estimate_p_Laplacian}})
\label{lemma: auxiliary lemma II}
    Let $2\leq p<\infty$, then there exists $C>0$ such that
    \[
        ||x|^{p-2}x-|y|^{p-2}y|\leq C|x-y|(|x|+|y|)^{p-2}
    \]
    for all $x,y\in\R$.
\end{lemma}

\begin{proposition}[Measurable UCP for anisotropic fractional $p$\,-biharmonic operator]
\label{prop: measurable UCP}
    Let $\Omega\subset \R^n$ be an open set, $0<s<2$, $2< p<\infty$ and one of the following conditions hold 
    \begin{enumerate}[(i)]
        \item\label{item n1 meas UCP} $n=1$, $0<s\leq 1+2/p$
        \item\label{item n2 meas UCP} $n=2$, $0<s\leq 4/p$
        \item\label{item n3 meas UCP} $n\geq 3$, $2<p<2^{\ast}=\frac{2n}{n-2}$, $0< s\leq 2(1+n(\frac{1}{p}-\frac{1}{2}))$.
    \end{enumerate} 
    If there exists a measurable subset $\Omega'\subset \Omega$ of positive measure and $u\in H^{1+s,p}(\R^n)$ satisfies
        \[
            (-\Delta)^{s}_pu=0\quad\text{in}\quad \Omega\quad\text{and}\quad (-\Delta)^{s/2}u=0\quad\text{in}\quad\Omega',
        \]
        then $u\equiv 0$ in $\R^n$.
\end{proposition}
\begin{proof}
    Without loss of generality, we can assume that $\Omega'$ is a compact set of positive measure and $\Omega$ is precompact. By mapping properties of the fractional Laplacian, we have $v=(-\Delta)^{s/2}u\in W^{1,p}(\R^n)$ and using H\"older's inequality, we deduce $|v|^{p-2}v\in L^{p'}(\R^n)$. We claim that $|v|^{p-2}v\in W^{1,p'}(\R^n)$. This follows by using standard methods, but for the convenience of the reader, we give here some details of the argument. Fix a sequence of standard mollifiers $\rho_{\epsilon}\in C_c^{\infty}(\R^n)$ and define $v_{\epsilon}=\rho_{\epsilon}\ast v\in C^{\infty}_b(\R^n)\cap W^{1,p}(\R^n)$. The function $f(x)= |x|^{p-2}x$ is of class $C^1$ with derivative $f'(x)=(p-1)|x|^{p-2}$. Therefore, by the the chain rule and integration by parts, there holds
    \[
        \int_{\R^n}|v_{\epsilon}|^{p-2}v_{\epsilon}\partial_i\phi\,dx=-(p-1)\int_{\R^n}|v_{\epsilon}|^{p-2}\partial_iv_{\epsilon} \phi\,dx
    \]
    for all $\epsilon>0$, $\phi\in C_c^{\infty}(\R^n)$ and $1\leq i\leq n$. Since $v_{\epsilon}\to v$ in $L^p(\R^n)$, we deduce by H\"older's inequality with $\frac{1}{p'}=\frac{p-2}{p}+\frac{1}{p}$ and Lemma~\ref{lemma: auxiliary lemma II} that
    \[
    \begin{split}
        &\left|\int_{\R^n}|v_{\epsilon}|^{p-2}v_{\epsilon}\partial_i\phi\,dx-\int_{\R^n}|v|^{p-2}v\partial_i\phi\,dx\right|\leq \||v_{\epsilon}|^{p-2}v_{\epsilon}-|v|^{p-2}v\|_{L^{p'}(\R^n)}\|\partial_i\phi\|_{L^p(\R^n)}\\
        &\leq C\|v_{\epsilon}-v\|_{L^p(\R^n)}(\|v_{\epsilon}\|_{L^p(\R^n)}+\|v\|_{L^p(\R^n)})^{p-2}\|\partial_i\phi\|_{L^p(\R^n)}\to 0
    \end{split}
    \]
    as $\epsilon\to 0$. Similarly, we obtain
    \[
    \begin{split}
        &\left|\int_{\R^n}|v_{\epsilon}|^{p-2}\partial_iv_{\epsilon}\phi\,dx-\int_{\R^n}|v|^{p-2}\partial_iv\phi\,dx\right|=\left|\int_{\R^n}(|v_{\epsilon}|^{p-2}(\partial_iv_{\epsilon}-\partial_i v)-(|v|^{p-2}-|v_{\epsilon}|^{p-2})\partial_iv)\phi\,dx\right|\\
        &\leq (\|v_{\epsilon}\|_{L^p(\R^n)}^{p-2}\|\partial_iv_{\epsilon}-\partial_iv\|_{L^p(\R^n)}+\||v_{\epsilon}|^{p-2}-|v|^{p-2}\|_{L^{\frac{p}{p-2}}(\R^n)})\|\partial_iv\|_{L^p(\R^n)}\|\phi\|_{L^p(\R^n)}.
    \end{split}
    \]
Since $v\in W^{1,p}(\R^n)$ we have $v_{\epsilon}\to v$ in $W^{1,p}(\R^n)$ and thus the first term vanishs as $\epsilon\to 0$. The second term converges to zero by the Radon--Riesz theorem (cf. \cite[Chapter 1, Theorem 1]{EvansWeakConv}). Hence, we have proved that 
\[
    \partial_i(|v|^{p-2}v)=(p-1)|v|^{p-2}\partial_i v\in L^{p'}(\R^n),
\]
which in turn implies $|v|^{p-2}v\in W^{1,p'}(\R^n)$. By \ref{item n1 meas UCP}--\ref{item n3 meas UCP} we have
\[
    1<p'<2 \quad\text{and}\quad \frac{s}{2}-\frac{n}{2}\leq 1-\frac{n}{p'}
\]
and therefore the Sobolev embedding shows $W^{1,p'}(\R^n)\hookrightarrow H^{s/2}(\R^n)$. Therefore, we can apply \cite[Proposition 5.1, Remark 5.6]{GRSU-fractional-calderon-single-measurement} to deduce that $v\equiv 0$ in $\R^n$ and therefore $(-\Delta)^{s/2}u\equiv 0$ in $\R^n$. Now one can repeat the argument in the proof of Theorem~\ref{thm: UCP} to deduce that $\text{supp}(\ifourier u)\subset \{0\}$ and, therefore, by the integrability assumption of $u$ that $u\equiv 0$ in $\R^n$.
\end{proof}

\section{Inverse problem for the anisotropic fractional $p$\,-biharmonic operator under monotonicity assumptions}\label{sec: inverse problems main}

In this section, we prove uniqueness results for the inverse problem related to anisotropic fractional $p$\,-biharmonic operators under a monotonicity assumption. 

\subsection{Setup of the inverse problem}
\label{subsec: Setup}

From now on let $\Omega\subset\R^n$ be a given bounded open set, $m\in\N$ and assume $A\in\mathbb{S}_+^m$ is a given matrix valued function.

\begin{definition}
\label{def: rescaled quantities}
    Let $\sigma\in L^{\infty}(\R^n)$ satisfy $\sigma(x)\geq \sigma_0>0$ a.e. in $\R^n$. Then we introduce the following rescaled quantities 
    \[
    \begin{split}
        &\mathcal{E}_{p,\sigma}\colon H^{s,p}(\R^n;\R^m)\to \R_+,\quad \mathcal{A}_{p,\sigma}\colon H^{s,p}(\R^n;\R^m)\times H^{s,p}(\R^n;\R^m)\to \R,\\
        &(-\Delta)^s_{p,\sigma}\colon H^{s,p}(\R^n;\R^m)\to (H^{s,p}(\R^n;\R^m))^*,\quad \Lambda_{\sigma}\colon X_p\to X_p^*
    \end{split}
    \]
   by
    \[
    \begin{split}
       &\mathcal{E}_{p,\sigma}(u)\vcentcolon =\mathcal{E}_{p,\sigma^{2/p}A}(u)=\frac{1}{p}\int_{\R^n}\sigma|A^{1/2}(-\Delta)^{s/2}u|^p\,dx\\
     & \mathcal{A}_{p,\sigma}(u,v)\vcentcolon = \int_{\R^n}\sigma|A^{1/2}(-\Delta)^{s/2}u|^{p-2}A(-\Delta)^{s/2}u\cdot (-\Delta)^{s/2}v\,dx\\
       &\langle (-\Delta)^s_{p,\sigma}u,v\rangle\vcentcolon = \mathcal{A}_{p,\sigma}(u,v)\quad\text{and}\quad \langle \Lambda_{\sigma}f,g\rangle\vcentcolon =\mathcal{A}_{p,\sigma}(u_f,g)
    \end{split}
    \]
    for all $u,v\in H^{s,p}(\R^n;\R^m)$ and $f,g\in X_p$, where $u_f$ is the unique solution of 
    \[
        \begin{split}
            (-\Delta)^s_{p,\sigma}u&=0,\quad \text{in}\quad \Omega,\\
            u&=f,\quad \text{in}\quad \Omega_e
        \end{split}
    \]
    (cf.~Section~\ref{sec: existence theory} and \ref{subsec: trace space and DN maps}).
\end{definition}

\begin{question}[Inverse problem]
    Let $\sigma\in L^{\infty}(\R^n)$ satisfy $\sigma(x)\geq \sigma_0>0$. Can we uniquely determine $\sigma$ in $\R^n$ from the knowledge of the nonlinear DN map $\Lambda_{\sigma}$ under some mild structural conditions on $\sigma$? 
\end{question}

\subsection{Pointwise two sided estimate for difference of DN maps}
\label{subsec: Pointwise two sided estimate for difference of DN maps}

\begin{lemma}
\label{needed_lemma}
Let $1< p<\infty$, $s> 0$ and assume that $\sigma_1,\sigma_2\in L^{\infty}(\R^n)$ satisfy $\sigma_1(x),\sigma_2(x)\geq \sigma_0>0$ in $\R^n$. If $u_0 \in X_p$, then 
\begin{align*}
    (p-1) & \int_{\R^n} \frac{\sigma_2}{\sigma_1^{1/(p-1)}} (\sigma_1^{\frac{1}{p-1}} - \sigma_2^{\frac{1}{p-1}}) |A^{1/2}(-\Delta)^{s/2} u_2|^p \,dx \\
    & \leq \langle(\Lambda_{\sigma_1} - \Lambda_{\sigma_2})u_0, u_0\rangle 
     \leq \int_{\R^n} (\sigma_1 - \sigma_2) |A^{1/2}(-\Delta)^{s/2} u_2|^p \,dx,
\end{align*}
where $u_2 \in  H^{s,p}(\R^n;\R^m)$ uniquely solves 
    \begin{equation}
    \label{eq: PDE exterior condition}
        \begin{split}
            (-\Delta)^s_{p,\sigma_2}u_2&=0,\quad\,\,\, \text{in}\quad \Omega,\\
            u_2&=u_0,\quad \text{in}\quad \Omega_e.
        \end{split}
    \end{equation}
    \end{lemma}
    
    \begin{remark}
        We emphasize that if $\sigma_1 \geq \sigma_2$, then all the terms in the inequality are nonnegative, while if $\sigma_1 \leq \sigma_2$, then they are nonpositive.
    \end{remark}
    
\begin{proof}
    Let $u_1, u_2 \in H^{s,p}(\R^n;\R^m)$ be the unique solutions of the exterior value problems
    \begin{equation}
    \label{eq: PDE exterior condition lemma 7.1 }
        \begin{split}
            (-\Delta)^s_{p,\sigma_i}u&=0,\quad\,\,\, \text{in}\quad \Omega,\\
            u&=u_0,\quad \text{in}\quad \Omega_e
        \end{split}
    \end{equation}
    for $i=1,2$. Note that the solution of \eqref{eq: PDE exterior condition lemma 7.1 } can be characterized as the unique minimizer of the energy functional $\mathcal{E}_{p,\sigma_i}$ over the affine subspace $\widetilde{H}^{s,p}_{u_0}(\Omega;\R^m)$.
    Therefore, by Lemma~\ref{lemma: DN maps}, we obtain the following one sided inequality for the difference of DN maps:
    \begin{align*}
    	\langle(\Lambda_{\sigma_1} - \Lambda_{\sigma_2})u_0, u_0\rangle 
	 &= \mathcal{A}_{p,\sigma_1}(u_1,u_0) - \mathcal{A}_{p,\sigma_2}(u_2,u_0)\\
    	& =\int_{\R^n}\sigma_1|A^{1/2}(-\Delta)^{s/2}u_1|^{p-2}A(-\Delta)^{s/2}u_1\cdot (-\Delta)^{s/2}u_1\,dx \\  &\quad- \int_{\R^n}\sigma_2|A^{1/2}(-\Delta)^{s/2}u_2|^{p-2}A(-\Delta)^{s/2}u_2\cdot (-\Delta)^{s/2}u_2\,dx\\
    	& = \int_{\R^n}\sigma_1|A^{1/2}(-\Delta)^{s/2}u_1|^p \,dx - \int_{\R^n}\sigma_2|A^{1/2}(-\Delta)^{s/2}u_2|^p \,dx \\
    	& \leq \int_{\R^n}(\sigma_1-\sigma_2)|A^{1/2}(-\Delta)^{s/2}u_2|^p \,dx.
\end{align*}
    Next we show the lower bound. Let $\beta > 0$ be a real number whose value will be fixed later. Using the definition of DN map several times together with the fact that $u_1|_{\Omega_e} = u_2|_{\Omega_e}=u_0$, we may rewrite the difference of the DN maps as follows:
\begin{align*}
	&\langle(\Lambda_{\sigma_1} - \Lambda_{\sigma_2})u_0, u_0\rangle\\
	&\quad= \int_{\R^n}\sigma_1|A^{1/2}(-\Delta)^{s/2}u_1|^{p-2}A(-\Delta)^{s/2}u_1\cdot (-\Delta)^{s/2}u_1\,dx  \\&\quad\quad- \int_{\R^n}\sigma_2|A^{1/2}(-\Delta)^{s/2}u_2|^{p-2}A(-\Delta)^{s/2}u_2\cdot (-\Delta)^{s/2}u_2\,dx\\
	&\quad= \int_{\R^n}\beta\sigma_2|A^{1/2}(-\Delta)^{s/2} u_2|^p \,dx\\
	&\quad\quad-\int_{\R^n}\left((1 + \beta)\sigma_2|A^{1/2}(-\Delta)^{s/2} u_2|^{p-2}A^{1/2}(-\Delta)^{s/2} u_2 \cdot A^{1/2}(-\Delta)^{s/2} u_2  - \sigma_1 |A^{1/2}(-\Delta)^{s/2} u_1|^p\right)  \,dx\\
	&\quad= \int_{\R^n}\beta\sigma_2|A^{1/2}(-\Delta)^{s/2}u_2|^p \,dx\\
	&\quad\quad-\int_{\R^n}\left((1+\beta)\sigma_2|A^{1/2}(-\Delta)^{s/2} u_2|^{p-2}A^{1/2}(-\Delta)^{s/2} u_2\cdot A^{1/2}(-\Delta)^{s/2} u_1- \sigma_1|A^{1/2}(-\Delta)^{s/2} u_1|^p\right)\,dx.
\end{align*}
In the last step, we used that $u_1$ and $u_2$ have the same exterior value $u_0$. Now, by applying Young's inequality $\abs{ab} \leq \abs{a}^p/p + \abs{b}^{p'}/p'$, we have 
\begin{align*}
	&(1+\beta)\sigma_2|A^{1/2}(-\Delta)^{s/2} u_2|^{p-2}A^{1/2}(-\Delta)^{s/2} u_2\cdot A^{1/2}(-\Delta)^{s/2} u_1 - \sigma_1|A^{1/2}(-\Delta)^{s/2} u_1|^p \\
	& = \frac{1+\beta}{p^{1/p}}\frac{\sigma_2}{\sigma_{1}^{1/p}}|A^{1/2}(-\Delta)^{s/2} u_2|^{p-2}A^{1/2}(-\Delta)^{s/2} u_2 \cdot  p^{1/p}\sigma_{1}^{1/p}A^{1/2}(-\Delta)^{s/2} u_1 - \sigma_1|A^{1/2}(-\Delta)^{s/2} u_1|^p \\
	& \leq \frac{1}{p'}\left(\frac{1+\beta}{p^{1/p}}\right)^{p'}\frac{\sigma_{2}^{p'}}{\sigma_{1}^{1/(p-1)}}|A^{1/2}(-\Delta)^{s/2} u_2|^p + \sigma_1|A^{1/2}(-\Delta)^{s/2} u_1|^p - \sigma_1|A^{1/2}(-\Delta)^{s/2} u_1|^p \\
	&= \frac{1}{p'}\left(1+\beta\right)^{p'}\frac{1}{p^{1/(p-1)}}\frac{\sigma_{2}^{p'}}{\sigma_{1}^{1/(p-1)}}|A^{1/2}(-\Delta)^{s/2} u_2|^p.
\end{align*}

Therefore, we obtain the lower bound
\begin{align}
	\scalemath{1}{ \langle(\Lambda_{\sigma_1} - \Lambda_{\sigma_2})u_0, u_0\rangle}\,
	&\scalemath{1}{ \geq \int_{\R^n}\left(\beta\sigma_2 - \frac{1}{p'}\left(1+\beta\right)^{p'}\frac{1}{p^{1/(p-1)}}\frac{\sigma_{2}^{p'}}{\sigma_{1}^{1/(p-1)}}\right)|A^{1/2}(-\Delta)^{s/2} u_2|^pdx} \nonumber \\
	& \scalemath{1}{= \int_{\R^n}\frac{\beta\sigma_2}{\sigma_{1}^{1/(p-1)}}\left(\sigma_{1}^{\frac{1}{p-1}} - \frac{1}{p'}\frac{(1+\beta)^{p'}}{\beta}\left(\frac{1}{p}\right)^{\frac{1}{p-1}}\sigma_{2}^{\frac{1}{p-1}}\right)|A^{1/2}(-\Delta)^{s/2} u_2|^pdx. }\label{DNkey}
\end{align}

Note that $\frac{(1+\beta)^{p'}}{\beta} \rightarrow \infty$ as $\beta\rightarrow \infty$ or $\beta\rightarrow 0$. 
So, the function $\beta \rightarrow \frac{(1+\beta)^{p'}}{\beta}$ attains its minimum at $\beta = p-1$.
Thus, we choose $\beta = p-1$ so that from \eqref{DNkey}, we obtain the required inequality. 
\end{proof}

\subsection{Uniqueness results}
\label{subsec: uniqueness results}

\begin{proof}[Proof of Theorem \ref{main_theorem}]
Without loss of generality, we can assume $D\setminus W\neq \emptyset$ as otherewise there is nothing to prove. We show the result by a contradiction argument. Let us consider a point $x_0\in D\setminus W$ and suppose that $\sigma_1(x_0) > \sigma_2(x_0)$. By assumption $\sigma_1-\sigma_2$ is lower semicontinuous in $D$ which means that the superlevel sets $\{\sigma_1-\sigma_2>a\}$, $a\in\R$, are open, but then this implies that there exists some open ball $B_r(x_0) \subset D$ such that $\sigma_1 - \sigma_2 > 0$ in $B_r(x_0)$. 

Next let $u_2\in  H^{s,p}(\R^n;\R^m)$ be the unique solution of
\begin{equation}
    \label{eq: PDE exterior condition theorem 7.2 }
        \begin{split}
            (-\Delta)^s_{p,\sigma_2}u_2&=0,\quad \text{in}\quad \Omega\\
            u_2&=u_0,\quad \text{in}\quad \Omega_e.
        \end{split}
    \end{equation} 
Up to shrinking the ball $B_r(x_0)$, we can assume that $B_r(x_0)\subset D\setminus \supp(u_0)$ since $\text{dist}(\partial W, \supp(u_0))>0$ and by the minimum principle for lower semicontinuous functions that there holds 
\[
    \sigma_{1}^{\frac{1}{p-1}} - \sigma_{2}^{\frac{1}{p-1}} \geq c_0 > 0\quad\text{in}\quad B_r(x_0),
\]
which can be applied as $g\colon\R_+\to\R_+,\,g(t)\vcentcolon = t^{1/(p-1)}$ is a nondecreasing continuous function and so $\sigma_{1}^{1/(p-1)} - \sigma_{2}^{1/(p-1)}$ is still a lower semicontinuous function. Using the left hand side of the monotonicity inequality (Lemma~\ref{needed_lemma}), the assumptions on $\sigma_1,\sigma_2$ and $A\in\mathbb{S}_+^m$ as well as $\Lambda_{\sigma_1}u_0|_W = \Lambda_{\sigma_2}u_0|_W$, we deduce that
\begin{equation}\label{eq2}
 \begin{split}
     \int_{B_r(x_0)}|(-\Delta)^{s/2} u_2|^{p} dx&\leq C(p-1)\int_{B_r(x_0)} \frac{\sigma_2}{\sigma_1^{1/(p-1)}} (\sigma_1^{\frac{1}{p-1}} - \sigma_2^{\frac{1}{p-1}}) |A^{1/2}(-\Delta)^{s/2} u_2|^p \,dx\\
     &\leq C(p-1)\int_{\R^n} \frac{\sigma_2}{\sigma_1^{1/(p-1)}} (\sigma_1^{\frac{1}{p-1}} - \sigma_2^{\frac{1}{p-1}}) |A^{1/2}(-\Delta)^{s/2} u_2|^p \,dx\\
     &\leq C\langle(\Lambda_{\sigma_1} - \Lambda_{\sigma_2})u_0, u_0\rangle =0
 \end{split} 
 \end{equation}
 for some $C>0$. This implies $(-\Delta)^{s/2}u_2 = 0$ a.e. $B_r(x_0)\subset D\setminus\supp(u_0)$. 
 
 Now we distinguish two cases. If $x_0\in \overline{\Omega}$, then there exists $\rho>0$, $x_1\in\Omega$ such that $B_{\rho}(x_1)\subset \Omega\cap B_r(x_0)$. By \eqref{eq: PDE exterior condition theorem 7.2 }, \eqref{eq2} the function $u_2\in H^{s,p}(\R^n;\R^m)$ satisfies $(-\Delta)^s_{p,\sigma_2}u_2=0,\, (-\Delta)^{s/2}u_2=0$ in $B_{\rho}(x_1)$. Hence, the unique continuation principle (Theorem~\ref{thm: UCP}) implies $u_2=0$ in $\R^n$ which contradicts the assumption $u_0\neq 0$. On the other hand, if $x_0\in \Omega_e$, then we can shrink $B_r(x_0)$ such that $B_r(x_0)\subset \Omega_e\cap(D\setminus\supp(u_0))$ but in this set we have $u_2=0$ since $u_2=u_0$ in $\Omega_e$. Therefore, $u_2\in H^{s,p}(\R^n;\R^m)$ satisfies $(-\Delta)^{s/2}u_2=u_2=0$ in $B_r(x_0)$. Hence, we deduce $u_2=0$ in $\R^n$ by the unique continuation principle for the fractional Laplacian (Theorem~\ref{UCP}). This again contradicts the assumption $u_0\neq 0$ and we can conclude the proof.
\end{proof}

\begin{proof}[Proof of Theorem~\ref{main_theorem 2}]
    Let $W_1,W_2\subset\Omega_e$ be two disjoint open sets. Applying Theorem~\ref{main_theorem} on these two respective sets with $D=\R^n$ we obtain $\sigma_1=\sigma_2$ on $\R^n\setminus W_1$ and $\R^n\setminus W_2$. Since $W_1,W_2$ are disjoint, this implies $\sigma_1=\sigma_2$ in $\R^n$.
\end{proof}

\begin{corollary}
    Let $2< p<\infty, s>0$ with $s\notin 2\N$ satisfy one of the conditions \ref{item n1 meas UCP}--\ref{item n3 meas UCP} in Proposition~\ref{prop: measurable UCP}. Suppose that there is a nonempty open set $W\subset\Omega_e$, a nonzero $u_0\in C_c^{\infty}(W;\R^m)$ and a solution $u_2\in H^{1+s,p}(\R^n;\R^m)$ of \eqref{eq: PDE exterior condition}. Assume that $\sigma_1, \sigma_2 \in L^{\infty}(\R^n)$ satisfy $\sigma_1(x),\sigma_2(x)\geq \sigma_0>0$ and $\sigma_1\geq \sigma_2$ in $\R^n$. If $\langle \Lambda_{\sigma_1}u_0,u_0\rangle = \langle \Lambda_{\sigma_2}u_0,u_0\rangle $, then $\sigma_1 = \sigma_2$ a.e. in $ \Omega$.
\end{corollary}

\begin{proof}
    Suppose by contradiction that there is a set of positive measure $A\subset \Omega$ such that $\sigma_1>\sigma_2$ in $A$. By Lusin's theorem (cf.~\cite[Theorem 1.14]{EvansMI}), there is a compact set $K\subset A$ of positive measure such that $\sigma_1^{\frac{1}{p-1}}-\sigma_2^{\frac{1}{p-1}}>0$ is continous on $K$. By the minimum principle, we again have  
    \[
        \sigma_1^{\frac{1}{p-1}}-\sigma_2^{\frac{1}{p-1}}\geq c_0>0\quad\text{on}\quad K.
    \]
    Now repeating the proof of Theorem~\ref{main_theorem} gives $(-\Delta)^{s/2}u=0$ a.e. in $K$. Now we can apply Proposition~\ref{prop: measurable UCP} to conclude that $u_2=0$ in $\R^n$, which again condradicts the assumption $u_0\neq 0$.
\end{proof}

\bibliography{biharmref} 

\newcommand{\etalchar}[1]{$^{#1}$}
\begin{thebibliography}{DPFBLR18}

\bibitem[AKS62]{AKS62}
N.~Aronszajn, A.~Krzywicki, and J.~Szarski.
\newblock A unique continuation theorem for exterior differential forms on
  {R}iemannian manifolds.
\newblock {\em Ark. Mat.}, 4:417--453 (1962), 1962.

\bibitem[Ale87]{Alessandrini:1987}
Giovanni Alessandrini.
\newblock Critical points of solutions to the {$p$}-{L}aplace equation in
  dimension two.
\newblock {\em Boll. Un. Mat. Ital. A (7)}, 1(2):239--246, 1987.

\bibitem[Aro57]{A57}
N.~Aronszajn.
\newblock A unique continuation theorem for solutions of elliptic partial
  differential equations or inequalities of second order.
\newblock {\em J. Math. Pures Appl. (9)}, 36:235--249, 1957.

\bibitem[ARS21]{ARS21-OnFractVersionMurat}
Harbir Antil, Carlos~N. Rautenberg, and Armin Schikorra.
\newblock On a fractional version of a {M}urat compactness result and
  applications.
\newblock {\em SIAM J. Math. Anal.}, 53(3):3158--3187, 2021.

\bibitem[AS01]{Alessandrini:Sigalotti:2001}
G.~Alessandrini and M.~Sigalotti.
\newblock Geometric properties of solutions to the anisotropic {$p$}-{L}aplace
  equation in dimension two.
\newblock {\em Ann. Acad. Sci. Fenn. Math.}, 26(1):249--266, 2001.

\bibitem[BHKS18]{Brander-Harrach-Kar-Salo}
Tommi Brander, Bastian Harrach, Manas Kar, and Mikko Salo.
\newblock Monotonicity and enclosure methods for the {$p$}-{L}aplace equation.
\newblock {\em SIAM J. Appl. Math.}, 78(2):742--758, 2018.

\bibitem[BI87]{bi84}
B.~Bojarski and T.~Iwaniec.
\newblock {$p$}-harmonic equation and quasiregular mappings.
\newblock In {\em Partial differential equations ({W}arsaw, 1984)}, volume~19
  of {\em Banach Center Publ.}, pages 25--38. PWN, Warsaw, 1987.

\bibitem[BKS15]{Brander:Kar:Salo:2014}
Tommi Brander, Manas Kar, and Mikko Salo.
\newblock Enclosure method for the {$p$}-{L}aplace equation.
\newblock {\em Inverse Problems}, 31(4):045001, 16, 2015.

\bibitem[BKS22]{Banerjee:Krishnan:Senapati:2022}
Agnid Banerjee, Venkateswaran~P. Krishnan, and Soumen Senapati.
\newblock The {C}alderón problem for space-time fractional parabolic operators
  with variable coefficients.
\newblock 2022.
\newblock arXiv:2205.12509.

\bibitem[BLP14]{MR3264796}
L.~Brasco, E.~Lindgren, and E.~Parini.
\newblock The fractional {C}heeger problem.
\newblock {\em Interfaces Free Bound.}, 16(3):419--458, 2014.

\bibitem[BP12]{ConvexityOptimization}
Viorel Barbu and Teodor Precupanu.
\newblock {\em Convexity and optimization in {B}anach spaces}.
\newblock Springer Monographs in Mathematics. Springer, Dordrecht, fourth
  edition, 2012.

\bibitem[BP16]{Brasco:Parini:2016}
Lorenzo Brasco and Enea Parini.
\newblock The second eigenvalue of the fractional {$p$}-{L}aplacian.
\newblock {\em Adv. Calc. Var.}, 9(4):323--355, 2016.

\bibitem[BPS16]{FracpLaplacian}
Lorenzo Brasco, Enea Parini, and Marco Squassina.
\newblock Stability of variational eigenvalues for the fractional
  {$p$}-{L}aplacian.
\newblock {\em Discrete Contin. Dyn. Syst.}, 36(4):1813--1845, 2016.

\bibitem[BPS22]{SpectraBiharm}
Davide Buoso, Luigi Provenzano, and Joachim Stubbe.
\newblock Semiclassical bounds for spectra of biharmonic operators.
\newblock {\em Rend. Mat. Appl. (7)}, 43(4):267--314, 2022.

\bibitem[Bra16]{Brander:2014}
Tommi Brander.
\newblock Calder\'on problem for the {$p$}-{L}aplacian: first order derivative
  of conductivity on the boundary.
\newblock {\em Proc. Amer. Math. Soc.}, 144(1):177--189, 2016.

\bibitem[Car39]{C39}
T.~Carleman.
\newblock Sur un probl\`eme d'unicit\'{e} pur les syst\`emes d'\'{e}quations
  aux d\'{e}riv\'{e}es partielles \`a deux variables ind\'{e}pendantes.
\newblock {\em Ark. Mat., Astr. Fys.}, 26(17):9, 1939.

\bibitem[CDLM20]{GenCKNThm}
Maria Colombo, Camillo De~Lellis, and Annalisa Massaccesi.
\newblock The generalized {C}affarelli-{K}ohn-{N}irenberg theorem for the
  hyperdissipative {N}avier-{S}tokes system.
\newblock {\em Comm. Pure Appl. Math.}, 73(3):609--663, 2020.

\bibitem[CDV22]{CDV-Bernstein-technique}
Xavier Cabr\'{e}, Serena Dipierro, and Enrico Valdinoci.
\newblock The {B}ernstein technique for integro-differential equations.
\newblock {\em Arch. Ration. Mech. Anal.}, 243(3):1597--1652, 2022.

\bibitem[CEFP{\etalchar{+}}21]{Corbo-Antonio}
Antonio Corbo~Esposito, Luisa Faella, Gianpaolo Piscitelli, Ravi Prakash, and
  Antonello Tamburrino.
\newblock Monotonicity principle in tomography of nonlinear conducting
  materials.
\newblock {\em Inverse Problems}, 37(4):Paper No. 045012, 25, 2021.

\bibitem[CF79]{CaffBiharmObstacle}
Luis~A. Caffarelli and Avner Friedman.
\newblock The obstacle problem for the biharmonic operator.
\newblock {\em Ann. Scuola Norm. Sup. Pisa Cl. Sci. (4)}, 6(1):151--184, 1979.

\bibitem[CG99]{Colombini:Grammatico:1999}
F.~Colombini and C.~Grammatico.
\newblock Some remarks on strong unique continuation for the {L}aplace operator
  and its powers.
\newblock {\em Comm. Partial Differential Equations}, 24(5-6):1079--1094, 1999.

\bibitem[CGFR21]{covietal2021calderon-directionally-antilocal}
Giovanni Covi, María~\'Angeles García-Ferrero, and Angkana Rüland.
\newblock On the {C}alder\'on problem for nonlocal {S}chr\"odinger equations
  with homogeneous, directionally antilocal principal symbols.
\newblock 2021.
\newblock arXiv:2109.14976.

\bibitem[CK10]{colombini:Koch:2010}
Ferruccio Colombini and Herbert Koch.
\newblock Strong unique continuation for products of elliptic operators of
  second order.
\newblock {\em Trans. Amer. Math. Soc.}, 362(1):345--355, 2010.

\bibitem[CL03]{MR1968344}
Zu-Chi Chen and Tao Luo.
\newblock The eigenvalue problem for the {$p$}-{L}aplacian-like equations.
\newblock {\em Int. J. Math. Math. Sci.}, (9):575--586, 2003.

\bibitem[CMR21]{CMR20}
Giovanni Covi, Keijo M\"{o}nkk\"{o}nen, and Jesse Railo.
\newblock Unique continuation property and {P}oincar\'{e} inequality for higher
  order fractional {L}aplacians with applications in inverse problems.
\newblock {\em Inverse Probl. Imaging}, 15(4):641--681, 2021.

\bibitem[CMRU22]{CMRU20-higher-order-fracCald}
Giovanni Covi, Keijo M\"{o}nkk\"{o}nen, Jesse Railo, and Gunther Uhlmann.
\newblock The higher order fractional {C}alder\'{o}n problem for linear local
  operators: uniqueness.
\newblock {\em Adv. Math.}, 399:Paper No. 108246, 29, 2022.

\bibitem[Cov21]{Covi:2021}
Giovanni Covi.
\newblock Uniqueness for the fractional {C}alder{\'o}n problem with quasilocal
  perturbations.
\newblock 2021.
\newblock arxiv:2110.11063.

\bibitem[CRZ22]{CRZ22}
Giovanni Covi, Jesse Railo, and Philipp Zimmermann.
\newblock The global inverse fractional conductivity problem.
\newblock 2022.
\newblock arxiv:2204.04325.

\bibitem[CS07]{CS-extension-problem-fractional-laplacian}
Luis Caffarelli and Luis Silvestre.
\newblock An extension problem related to the fractional {L}aplacian.
\newblock {\em Comm. Partial Differential Equations}, 32(7-9):1245--1260, 2007.

\bibitem[CS14]{CS-nonlinera-equations-fractional-laplacians}
Xavier Cabr\'{e} and Yannick Sire.
\newblock Nonlinear equations for fractional {L}aplacians, {I}: {R}egularity,
  maximum principles, and {H}amiltonian estimates.
\newblock {\em Ann. Inst. H. Poincar\'{e} Anal. Non Lin\'{e}aire},
  31(1):23--53, 2014.

\bibitem[DLS17]{RegularityHarmonicMaps}
Francesca Da~Lio and Armin Schikorra.
\newblock On regularity theory for n/p-harmonic maps into manifolds.
\newblock {\em Nonlinear Analysis}, 165:182–197, Dec 2017.

\bibitem[DNPV12]{DINEPV-hitchhiker-sobolev}
Eleonora Di~Nezza, Giampiero Palatucci, and Enrico Valdinoci.
\newblock Hitchhiker's guide to the fractional {S}obolev spaces.
\newblock {\em Bull. Sci. Math.}, 136(5):521--573, 2012.

\bibitem[DPFBLR18]{PBL18-Opt-Prob-First-Eigenvalues-pFracLap}
Leandro Del~Pezzo, Juli\'{a}n Fern\'{a}ndez~Bonder, and Luis
  L\'{o}pez~R\'{\i}os.
\newblock An optimization problem for the first eigenvalue of the
  {$p$}-fractional {L}aplacian.
\newblock {\em Math. Nachr.}, 291(4):632--651, 2018.

\bibitem[DQ20]{DELPEZZO2020111479}
Leandro~M. {Del Pezzo} and Alexander Quaas.
\newblock Spectrum of the fractional p-laplacian in rn and decay estimate for
  positive solutions of a schrödinger equation.
\newblock {\em Nonlinear Analysis}, 193:111479, 2020.
\newblock Nonlocal and Fractional Phenomena.

\bibitem[DSFKSU09]{DSFKSU09}
David Dos Santos~Ferreira, Carlos~E. Kenig, Johannes Sj\"{o}strand, and Gunther
  Uhlmann.
\newblock On the linearized local {C}alder\'{o}n problem.
\newblock {\em Math. Res. Lett.}, 16(6):955--970, 2009.

\bibitem[DSV17]{DSV-all-functions-are-s-harmonic}
Serena Dipierro, Ovidiu Savin, and Enrico Valdinoci.
\newblock All functions are locally {$s$}-harmonic up to a small error.
\newblock {\em J. Eur. Math. Soc. (JEMS)}, 19(4):957--966, 2017.

\bibitem[dTGCV21]{ThreeRep}
F\'{e}lix del Teso, David G\'{o}mez-Castro, and Juan~Luis V\'{a}zquez.
\newblock Three representations of the fractional {$p$}-{L}aplacian: semigroup,
  extension and {B}alakrishnan formulas.
\newblock {\em Fract. Calc. Appl. Anal.}, 24(4):966--1002, 2021.

\bibitem[EG15]{EvansMI}
Lawrence~C. Evans and Ronald~F. Gariepy.
\newblock {\em Measure theory and fine properties of functions}.
\newblock Textbooks in Mathematics. CRC Press, Boca Raton, FL, revised edition,
  2015.

\bibitem[EK11]{ClassPBiharm}
Abdelouahed El~Khalil.
\newblock On a class of {PDE} involving {$p$}-biharmonic operator.
\newblock {\em ISRN Math. Anal.}, pages Art. ID 630745, 11, 2011.

\bibitem[EKKT02]{Spectrum-p-Biharm}
Abdelouahed El~Khalil, Siham Kellati, and Abdelfattah Touzani.
\newblock On the spectrum of the {$p$}-biharmonic operator.
\newblock In {\em Proceedings of the 2002 {F}ez {C}onference on {P}artial
  {D}ifferential {E}quations}, volume~9 of {\em Electron. J. Differ. Equ.
  Conf.}, pages 161--170. Southwest Texas State Univ., San Marcos, TX, 2002.

\bibitem[ENT13]{Esposito:Nitsch:Trombetti:2013}
L.~Esposito, C.~Nitsch, and C.~Trombetti.
\newblock Best constants in {P}oincar\'{e} inequalities for convex domains.
\newblock {\em J. Convex Anal.}, 20(1):253--264, 2013.

\bibitem[Eva90]{EvansWeakConv}
Lawrence~C. Evans.
\newblock {\em Weak convergence methods for nonlinear partial differential
  equations}, volume~74 of {\em CBMS Regional Conference Series in
  Mathematics}.
\newblock Published for the Conference Board of the Mathematical Sciences,
  Washington, DC; by the American Mathematical Society, Providence, RI, 1990.

\bibitem[FF14]{FF-unique-continuation-fractional-ellliptic-equations}
Mouhamed~Moustapha Fall and Veronica Felli.
\newblock Unique continuation property and local asymptotics of solutions to
  fractional elliptic equations.
\newblock {\em Comm. Partial Differential Equations}, 39(2):354--397, 2014.

\bibitem[FF20]{FF-unique-continuation-higher-laplacian}
Veronica Felli and Alberto Ferrero.
\newblock Unique continuation principles for a higher order fractional
  {L}aplace equation.
\newblock {\em Nonlinearity}, 33(8):4133--4190, 2020.

\bibitem[FGKU21]{feizmohammadiEtAl2021fractional}
Ali Feizmohammadi, Tuhin Ghosh, Katya Krupchyk, and Gunther Uhlmann.
\newblock Fractional anisotropic {C}alder\'on problem on closed {R}iemannian
  manifolds.
\newblock 2021.
\newblock arXiv:2112.03480.

\bibitem[FKV15]{DirichletNonlocalOperators}
Matthieu Felsinger, Moritz Kassmann, and Paul Voigt.
\newblock The {D}irichlet problem for nonlocal operators.
\newblock {\em Mathematische Zeitschrift}, 279(3):779--809, 2015.

\bibitem[GFR19]{GR-fractional-laplacian-strong-unique-continuation}
Mar\'{\i}a~Angeles Garc\'{\i}a-Ferrero and Angkana R{\"u}land.
\newblock Strong unique continuation for the higher order fractional
  {L}aplacian.
\newblock {\em Math. Eng.}, 1(4):715--774, 2019.

\bibitem[GGS10]{Gazzola}
Filippo Gazzola, Hans-Christoph Grunau, and Guido Sweers.
\newblock {\em Polyharmonic boundary value problems}, volume 1991 of {\em
  Lecture Notes in Mathematics}.
\newblock Springer-Verlag, Berlin, 2010.
\newblock Positivity preserving and nonlinear higher order elliptic equations
  in bounded domains.

\bibitem[GHS21]{GHS20-Compact-embeddings-Besov-Triebel}
Helena~F. Gon\c{c}alves, Dorothee~D. Haroske, and Leszek Skrzypczak.
\newblock Compact embeddings in {B}esov-type and {T}riebel-{L}izorkin-type
  spaces on bounded domains.
\newblock {\em Rev. Mat. Complut.}, 34(3):761--795, 2021.

\bibitem[GK16]{Guo-Kar}
Chang-Yu Guo and Manas Kar.
\newblock Quantitative uniqueness estimates for {$p$}-{L}aplace type equations
  in the plane.
\newblock {\em Nonlinear Anal.}, 143:19--44, 2016.

\bibitem[GKS16]{Guo-Kar-Salo}
Chang-Yu Guo, Manas Kar, and Mikko Salo.
\newblock Inverse problems for {$p$}-{L}aplace type equations under
  monotonicity assumptions.
\newblock {\em Rend. Istit. Mat. Univ. Trieste}, 48:79--99, 2016.

\bibitem[GM75]{Estimate_p_Laplacian}
R.~Glowinski and A.~Marrocco.
\newblock Sur l'approximation, par \'{e}l\'{e}ments finis d'ordre un, et la
  r\'{e}solution, par p\'{e}nalisation-dualit\'{e}, d'une classe de probl\`emes
  de {D}irichlet non lin\'{e}aires.
\newblock {\em Rev. Fran\c{c}aise Automat. Informat. Recherche
  Op\'{e}rationnelle S\'{e}r. Rouge Anal. Num\'{e}r.}, 9({\rm R}-2):41--76,
  1975.

\bibitem[GM14]{Seppo:Niko:2014}
Seppo Granlund and Niko Marola.
\newblock On the problem of unique continuation for the {$p$}-{L}aplace
  equation.
\newblock {\em Nonlinear Anal.}, 101:89--97, 2014.

\bibitem[Gra14]{Grafakos1}
Loukas Grafakos.
\newblock {\em Classical {F}ourier analysis}, volume 249 of {\em Graduate Texts
  in Mathematics}.
\newblock Springer, New York, third edition, 2014.

\bibitem[GRSU20]{GRSU-fractional-calderon-single-measurement}
Tuhin Ghosh, Angkana R{\"u}land, Mikko Salo, and Gunther Uhlmann.
\newblock Uniqueness and reconstruction for the fractional {C}alder\'{o}n
  problem with a single measurement.
\newblock {\em J. Funct. Anal.}, 279(1):108505, 42, 2020.

\bibitem[GSU20]{GSU20}
Tuhin Ghosh, Mikko Salo, and Gunther Uhlmann.
\newblock The {C}alder\'{o}n problem for the fractional {S}chr\"{o}dinger
  equation.
\newblock {\em Anal. PDE}, 13(2):455--475, 2020.

\bibitem[GU09]{WeightedSobolev}
V.~Gol{'}dshtein and A.~Ukhlov.
\newblock Weighted {S}obolev spaces and embedding theorems.
\newblock {\em Trans. Amer. Math. Soc.}, 361(7):3829--3850, 2009.

\bibitem[GU21]{ghosh2021calderon}
Tuhin Ghosh and Gunther Uhlmann.
\newblock The {C}alder\'{o}n problem for nonlocal operators.
\newblock 2021.
\newblock arXiv:2110.09265.

\bibitem[HHM19]{Hannukainen-Nuutti-Lauri}
Antti Hannukainen, Nuutti Hyv\"{o}nen, and Lauri Mustonen.
\newblock An inverse boundary value problem for the {$p$}-{L}aplacian: a
  linearization approach.
\newblock {\em Inverse Problems}, 35(3):034001, 24, 2019.

\bibitem[HKM93]{Heinonen:Kilpelainen:Martio:1993}
Juha Heinonen, Tero Kilpel{\"a}inen, and Olli Martio.
\newblock {\em Nonlinear potential theory of degenerate elliptic equations}.
\newblock Oxford Mathematical Monographs. The Clarendon Press, Oxford
  University Press, New York, Oxford, 1993.
\newblock Oxford Science Publications.

\bibitem[HL19]{Harrach:Lin:2019}
Bastian Harrach and Yi-Hsuan Lin.
\newblock Monotonicity-based inversion of the fractional {S}chr\"{o}dinger
  equation {I}. {P}ositive potentials.
\newblock {\em SIAM J. Math. Anal.}, 51(4):3092--3111, 2019.

\bibitem[HL20]{Harrach:Lin:2020}
Bastian Harrach and Yi-Hsuan Lin.
\newblock Monotonicity-based inversion of the fractional {S}ch\"{o}dinger
  equation {II}. {G}eneral potentials and stability.
\newblock {\em SIAM J. Math. Anal.}, 52(1):402--436, 2020.

\bibitem[HLYZ20]{Helin:Lassas:Ylinen:Zhang:2020}
Tapio Helin, Matti Lassas, Lauri Ylinen, and Zhidong Zhang.
\newblock Inverse problems for heat equation and space-time fractional
  diffusion equation with one measurement.
\newblock {\em J. Differential Equations}, 269(9):7498--7528, 2020.

\bibitem[H{\"o}r03]{HO:analysis-of-pdos}
Lars H{\"o}rmander.
\newblock {\em The analysis of linear partial differential operators. {I}}.
\newblock Classics in Mathematics. Springer-Verlag, Berlin, 2003.
\newblock Distribution theory and Fourier analysis, Reprint of the second
  (1990) edition [Springer, Berlin; MR1065993 (91m:35001a)].

\bibitem[HU13]{Harrach:Ullrich:2013}
Bastian Harrach and Marcel Ullrich.
\newblock Monotonicity-based shape reconstruction in electrical impedance
  tomography.
\newblock {\em SIAM J. Math. Anal.}, 45(6):3382--3403, 2013.

\bibitem[HYZ12]{HYZ12-FracGaglNirenHardy}
Hichem Hajaiej, Xinwei Yu, and Zhichun Zhai.
\newblock Fractional {G}agliardo-{N}irenberg and {H}ardy inequalities under
  {L}orentz norms.
\newblock {\em J. Math. Anal. Appl.}, 396(2):569--577, 2012.

\bibitem[IM20]{IM-unique-continuation-riesz-potential}
Joonas Ilmavirta and Keijo M\"{o}nkk\"{o}nen.
\newblock Unique continuation of the normal operator of the x-ray transform and
  applications in geophysics.
\newblock {\em Inverse Problems}, 36(4):045014, 23, 2020.

\bibitem[Isa07]{Is07}
Victor Isakov.
\newblock On uniqueness in the inverse conductivity problem with local data.
\newblock {\em Inverse Probl. Imaging}, 1(1):95--105, 2007.

\bibitem[Joh82]{John:1982}
Fritz John.
\newblock {\em Partial differential equations}, volume~1 of {\em Applied
  Mathematical Sciences}.
\newblock Springer-Verlag, New York, fourth edition, 1982.

\bibitem[KLW22]{Kow:Lin:Wang:2022}
Pu-Zhao Kow, Yi-Hsuan Lin, and Jenn-Nan Wang.
\newblock The {C}alder\'{o}n {P}roblem for the {F}ractional {W}ave {E}quation:
  {U}niqueness and {O}ptimal {S}tability.
\newblock {\em SIAM J. Math. Anal.}, 54(3):3379--3419, 2022.

\bibitem[KLY0]{Kian:Liu:Yamamoto:2022}
Yavar Kian, Yikan Liu, and Masahiro Yamamoto.
\newblock Uniqueness of inverse source problems for general evolution
  equations.
\newblock {\em Communications in Contemporary Mathematics}, 0(0):2250009, 0.

\bibitem[Kry19]{KR-all-functions-are-s-harmonic}
N.~V. Krylov.
\newblock All functions are locally {$s$}-harmonic up to a small error.
\newblock {\em J. Funct. Anal.}, 277(8):2728--2733, 2019.

\bibitem[KSU07]{KSU07}
Carlos~E. Kenig, Johannes Sj\"{o}strand, and Gunther Uhlmann.
\newblock The {C}alder\'{o}n problem with partial data.
\newblock {\em Ann. of Math. (2)}, 165(2):567--591, 2007.

\bibitem[KT01]{KT08}
Herbert Koch and Daniel Tataru.
\newblock Carleman estimates and unique continuation for second-order elliptic
  equations with nonsmooth coefficients.
\newblock {\em Comm. Pure Appl. Math.}, 54(3):339--360, 2001.

\bibitem[KW21]{Kar-Wang}
Manas Kar and Jenn-Nan Wang.
\newblock Size estimates for the weighted {$p$}-{L}aplace equation with one
  measurement.
\newblock {\em Discrete Contin. Dyn. Syst. Ser. B}, 26(4):2011--2024, 2021.

\bibitem[L{\^e}06]{EigenPLap2}
An~L{\^e}.
\newblock Eigenvalue problems for the {$p$}-{L}aplacian.
\newblock {\em Nonlinear Anal.}, 64(5):1057--1099, 2006.

\bibitem[Ler19]{Lerner:2019}
Nicolas Lerner.
\newblock {\em Carleman inequalities}, volume 353 of {\em Grundlehren der
  mathematischen Wissenschaften [Fundamental Principles of Mathematical
  Sciences]}.
\newblock Springer, Cham, 2019.
\newblock An introduction and more.

\bibitem[Lin22]{Lin:2020}
Yi-Hsuan Lin.
\newblock Monotonicity-based inversion of fractional semilinear elliptic
  equations with power type nonlinearities.
\newblock {\em Calc. Var. Partial Differential Equations}, 61(5):Paper No. 188,
  2022.

\bibitem[LL85]{Spectrum-Biharm}
Pui~Fai Leung and Luen~Chau Li.
\newblock On the spectrum of the biharmonic operator in a bounded domain.
\newblock {\em Bull. Austral. Math. Soc.}, 31(1):83--88, 1985.

\bibitem[LL14]{LL14--Fractional-eigenvalues}
Erik Lindgren and Peter Lindqvist.
\newblock Fractional eigenvalues.
\newblock {\em Calc. Var. Partial Differential Equations}, 49(1-2):795--826,
  2014.

\bibitem[LL22]{LL-fractional-semilinear-problems}
Ru-Yu Lai and Yi-Hsuan Lin.
\newblock Inverse problems for fractional semilinear elliptic equations.
\newblock {\em Nonlinear Anal.}, 216:Paper No. 112699, 2022.

\bibitem[LLR20]{Lai:Lin:Ruland:2020}
Ru-Yu Lai, Yi-Hsuan Lin, and Angkana R\"{u}land.
\newblock The {C}alder\'{o}n problem for a space-time fractional parabolic
  equation.
\newblock {\em SIAM J. Math. Anal.}, 52(3):2655--2688, 2020.

\bibitem[LNW11]{Lin:Sei:Wang:2011}
Ching-Lung Lin, Sei Nagayasu, and Jenn-Nan Wang.
\newblock Quantitative uniqueness for the power of the {L}aplacian with
  singular coefficients.
\newblock {\em Ann. Sc. Norm. Super. Pisa Cl. Sci. (5)}, 10(3):513--529, 2011.

\bibitem[LS14]{Liimatainen:Salo:2012}
Tony Liimatainen and Mikko Salo.
\newblock {$n$}-harmonic coordinates and the regularity of conformal mappings.
\newblock {\em Math. Res. Lett.}, 21:341--361, 2014.

\bibitem[Ly05]{EigenPLap}
Idrissa Ly.
\newblock The first eigenvalue for the {$p$}-{L}aplacian operator.
\newblock {\em JIPAM. J. Inequal. Pure Appl. Math.}, 6(3):Article 91, 12, 2005.

\bibitem[Man88]{Manfredi:1988}
Juan~J. Manfredi.
\newblock {$p$}-harmonic functions in the plane.
\newblock {\em Proc. Amer. Math. Soc.}, 103(2):473--479, 1988.

\bibitem[Mit18]{MI:distribution-theory}
Dorina Mitrea.
\newblock {\em Distributions, partial differential equations, and harmonic
  analysis}.
\newblock Universitext. Springer, Cham, 2018.
\newblock Second edition [ MR3114783].

\bibitem[MS21]{MS21-Best-frac-p-Poincare-unboundedDoms}
Kaushik Mohanta and Firoj Sk.
\newblock On the best constant in fractional {$p$}-{P}oincar\'{e} inequalities
  on cylindrical domains.
\newblock {\em Differential Integral Equations}, 34(11-12):691--712, 2021.

\bibitem[Rie38]{RI-liouville-riemann-integrals-potentials}
Marcel Riesz.
\newblock Int{\'e}grales de {R}iemann-{L}iouville et potentiels.
\newblock {\em Acta Sci. Math. Szeged}, 9(1-1):1--42, 1938.

\bibitem[RO16]{RosOton16-NonlocEllipticSurvey}
Xavier Ros-Oton.
\newblock Nonlocal elliptic equations in bounded domains: a survey.
\newblock {\em Publ. Mat.}, 60(1):3--26, 2016.

\bibitem[RS18]{RS17d}
Angkana {R}{ü}land and Mikko Salo.
\newblock Exponential instability in the fractional {C}alder\'{o}n problem.
\newblock {\em Inverse Problems}, 34(4):045003, 21, 2018.

\bibitem[RS20]{RS-fractional-calderon-low-regularity-stability}
Angkana R{\"u}land and Mikko Salo.
\newblock The fractional {C}alder\'{o}n problem: low regularity and stability.
\newblock {\em Nonlinear Anal.}, 193:111529, 56, 2020.

\bibitem[R{\"u}l15]{Ru15}
Angkana R{\"u}land.
\newblock Unique continuation for fractional {S}chr\"{o}dinger equations with
  rough potentials.
\newblock {\em Comm. Partial Differential Equations}, 40(1):77--114, 2015.

\bibitem[{R}{\"u}l21]{R-singular-measurement}
Angkana {R}{\"u}land.
\newblock On single measurement stability for the fractional {C}alder\'{o}n
  problem.
\newblock {\em SIAM J. Math. Anal.}, 53(5):5094--5113, 2021.

\bibitem[RZ22a]{RZ2022counterexamples}
Jesse Railo and Philipp Zimmermann.
\newblock Counterexamples to uniqueness in the inverse fractional conductivity
  problem with partial data.
\newblock 2022.
\newblock arXiv:2203.02442.

\bibitem[RZ22b]{RZ2022unboundedFracCald}
Jesse Railo and Philipp Zimmermann.
\newblock Fractional {C}alderón problems and {P}oincaré inequalities on
  unbounded domains.
\newblock 2022.
\newblock arxiv:2203.02425.

\bibitem[Sal17]{Sal17}
Mikko Salo.
\newblock {The fractional Calderón problem}.
\newblock {\em Journées équations aux dérivées partielles}, Exp. No.(7),
  2017.

\bibitem[Sim78]{Simon}
Jacques Simon.
\newblock R\'{e}gularit\'{e} de la solution d'une \'{e}quation non lin\'{e}aire
  dans {${\bf R}^{N}$}.
\newblock In {\em Journ\'{e}es d'{A}nalyse {N}on {L}in\'{e}aire ({P}roc.
  {C}onf., {B}esan\c{c}on, 1977)}, volume 665 of {\em Lecture Notes in Math.},
  pages 205--227. Springer, Berlin, 1978.

\bibitem[ST10]{StingaTorrea}
Pablo~Ra\'{u}l Stinga and Jos\'{e}~Luis Torrea.
\newblock Extension problem and {H}arnack's inequality for some fractional
  operators.
\newblock {\em Comm. Partial Differential Equations}, 35(11):2092--2122, 2010.

\bibitem[Ste70]{Singular-Integrals-Stein}
Elias~M. Stein.
\newblock {\em Singular integrals and differentiability properties of
  functions}.
\newblock Princeton Mathematical Series, No. 30. Princeton University Press,
  Princeton, N.J., 1970.

\bibitem[Str08]{Variational-Methods}
Michael Struwe.
\newblock {\em Variational methods}, volume~34 of {\em Ergebnisse der
  Mathematik und ihrer Grenzgebiete. 3. Folge. A Series of Modern Surveys in
  Mathematics [Results in Mathematics and Related Areas. 3rd Series. A Series
  of Modern Surveys in Mathematics]}.
\newblock Springer-Verlag, Berlin, fourth edition, 2008.
\newblock Applications to nonlinear partial differential equations and
  Hamiltonian systems.

\bibitem[SU87]{SU87}
John Sylvester and Gunther Uhlmann.
\newblock A global uniqueness theorem for an inverse boundary value problem.
\newblock {\em Ann. of Math. (2)}, 125(1):153--169, 1987.

\bibitem[SZ12]{Salo:Zhong:2012}
Mikko Salo and Xiao Zhong.
\newblock An inverse problem for the $p$-{L}aplacian: {B}oundary determination.
\newblock {\em {SIAM} J. Math. Anal.}, 44(4):2474--2495, March 2012.

\bibitem[TR02]{Tamburrino:Rubinacci:2002}
A.~Tamburrino and G.~Rubinacci.
\newblock A new non-iterative inversion method for electrical resistance
  tomography.
\newblock volume~18, pages 1809--1829. 2002.
\newblock Special section on electromagnetic and ultrasonic nondestructive
  evaluation.

\bibitem[Tri83]{Triebel-Theory-of-function-spaces}
H.~Triebel.
\newblock {\em Theory of function spaces}, volume~38 of {\em Mathematik und
  ihre Anwendungen in Physik und Technik [Mathematics and its Applications in
  Physics and Technology]}.
\newblock Akademische Verlagsgesellschaft Geest \& Portig K.-G., Leipzig, 1983.

\bibitem[TT07]{Talbi2007OnTS}
Mohamed Talbi and N.~Tsouli.
\newblock On the spectrum of the weighted p-biharmonic operator with weight.
\newblock {\em Mediterranean Journal of Mathematics}, 4:73--86, 2007.

\bibitem[Yam22]{Yamamoto:2022}
Masahiro Yamamoto.
\newblock Fractional calculus and time-fractional differential equations:
  Revisit and construction of a theory.
\newblock {\em Mathematics}, 10(5), 2022.

\bibitem[{Yan}13]{YA-higher-order-laplacian}
Ray {Yang}.
\newblock {On higher order extensions for the fractional {L}aplacian}.
\newblock 2013.
\newblock arXiv:1302.4413.

\end{thebibliography}

\bibliographystyle{alpha}

\end{document}